\documentclass[9pt]{amsart}
\usepackage{amsmath,amssymb,latexsym,amscd}
\usepackage[usenames]{color}
\usepackage{palatino}
\usepackage {amssymb}
\usepackage {amsmath}
\usepackage {graphicx}

\newcommand{\R}{{\Bbb R}}
\newcommand{\PP}{{\Bbb P}}
\newcommand{\T}{{\Bbb T}}
\newcommand{\N}{{\Bbb N}}
\newcommand{\D}{{\Bbb D}}
\newcommand{\C}{{\Bbb C}}
\newcommand{\Nf}{{\mathfrak N}}
\newcommand{\ep}{{\varepsilon}}
\newcommand{\p}{{\partial}}
\newcommand{\bp}{{\mathfrak p}}
\newcommand{\bpx}{{\mathfrak p^{-\bold x}}}
\newcommand{\bppx}{{\mathfrak p'^{-\bold x}}}
\newcommand{\Il}{{\mathcal I}}
\newcommand{\Jl}{{\mathcal J}}
\newcommand{\Wl}{{\mathcal W}}
\newcommand{\Hl}{{\mathcal H}}
\newcommand{\Xl}{{\mathcal X}}
\newcommand{\Yl}{{\mathcal Y}}

\newcommand{\Bl}{{\mathcal B}}
\newcommand{\Al}{{\mathcal A}}
\newcommand{\Tl}{{\mathcal T}}
\newcommand{\Ll}{{\mathcal L}}
\newcommand{\Dl}{{\mathcal D}}
\newcommand{\Pl}{{\mathcal P}}
\newcommand{\Kl}{{\mathcal K}}
\newcommand{\Sl}{{\mathcal S}}
\newcommand{\Xw}{{\mathcal X_{w}}}
\newcommand{\Hpa}{{\mathcal H^{p}_{w}}}
\newcommand{\Hw}{{\mathcal H^{2}_{w}}}
\newcommand{\Ba}{{\mathcal B^{2}_{\beta}}}

\newcommand{\Ab}{{\mathcal A^{2}_{\beta}}}
\def\d{\delta}
\def\b{\beta}

\def\a{\alpha}
\def\e{\eta}
\def\s{\sigma}
\def\ss{\sigma'}

\def\lo{\log}
\def\le{\lesssim}

\newtheorem{thm}{Theorem}
\newtheorem{prop}{Proposition}
\newtheorem{lem}{Lemma}
\newtheorem{cor}{Corollary}

\newtheorem{defi}{Definition}

\title{Volterra operators and Hankel forms  on Bergman spaces of Dirichlet series}

\author{H. Bommier-Hato}

\begin{document}
\address{Bommier-Hato: Faculty of Mathematics,
University of Vienna, Oskar-Morgenstern-Platz 1, 1090 Vienna, Austria}

\email{ helene.bommier@gmail.com}

\thanks{The author was supported by the FWF project P 30251-N35.}



\maketitle
\begin{abstract} 
For a Dirichlet series $g$, we study the Volterra operator $T_g f(s)=-\int^{+\infty}_{s} f(w)g'(w)dw,$ acting on a class of weighted Hilbert spaces $\Hw$ of Dirichlet series. We obtain sufficient / necessary conditions for $T_g$ to be bounded (resp. compact), involving BMO and Bloch type spaces on some half-plane. We also investigate the membership of $T_g$ in Schatten classes. Moreover, we show that if $T_g$ is bounded, then $g$ is in $\Hl^p_w$, the $L^p$-version of $\Hw$, for every $0<p<\infty$. We also relate the boundedness of $T_g$ to the boundedness of a multiplicative Hankel form of symbol $g$, and the membership of $g$ in the dual of $\Hl^1_w$.

\end{abstract}

\section{Introduction}
 Dirichlet series are functions of the form 
\begin{equation}\label{dirichlet series def}
	f(s)= \sum^{+\infty}_{n=1}a_n n^{-s},\ \text{ with }s\in\C. 
\end{equation}
 For a real number $\theta$, $\C_{\theta}$ stands for the half-plane $\left\{s,\ \Re s>\theta \right\}$, and $\D$ for the unit disk. $\Dl$ denotes the class of functions $f$ of the form (\ref{dirichlet series def}) in some  half-plane $\C_{\theta}$, and $\Pl$ is the space of Dirichlet polynomials.

 The increasing sequence of prime numbers will be denoted by $(p_j)_{j\geq 1}$, and the set of all primes by $\PP$. Given a positive integer $n$, $n=p^{\kappa}$ will stand for the prime number factorization 
$ n=p^{\kappa_1}_{1}p^{\kappa_2}_{2}\cdots p^{\kappa_d}_{d}, $
 which associates uniquely to $n$ the finite multi-index $\kappa(n)=(\kappa_1,\kappa_2,\cdots,\kappa_d)$. The number of prime factors in $n$ is denoted by $\Omega(n)$ (counting multiplicities), and by $\omega(n)$ (without multiplicities).

The space of eventually zero complex sequences $c_{00}$ consists in all sequences which have only finitely many non zero elements. We set $\D^{\infty}_{\text{fin}}=\D^{\infty}\cap c_{00}$ and $\N^{\infty}_{\text{0,fin}}=\N_0^{\infty}\cap c_{00}$, where $\N_0=\N\cup\left\{0\right\}$ is the set of non-negative integers. 

Let $F:\D^{\infty}_{\text{fin}}\rightarrow\C$ be analytic, i.e. analytic at every point $z\in\D^{\infty}_{\text{fin}}$ separately with respect to each variable. Then $F$ can be written as a convergent Taylor series 
$$  F(z)=\sum_{\a\in \N^{\infty}_{\text{0,fin}}}c_{\a} z^{\a},\ z\in\D^{\infty}_{\text{fin}}.$$
The truncation $A_m F$ of $F$ onto the first $m$ variables is defined by
$$ A_m F(z)=F(z_1,\cdots,z_m,0,0,\cdots).  $$
For $z,\chi$ in $ \D^{\infty}$, we set $z.\chi:=(z_1\chi_1,z_2\chi_2,\cdots)$, and  $\bp^{\bold x}:=(p_1^x, p_2^x,\cdots)$ for a real number $x$, .

 The Bohr  lift \cite{bohr} of the Dirichlet series $f(s)= \sum^{+\infty}_{n=1}a_n n^{-s}$ is the power series 
\begin{equation*}
	\mathcal{B}f(\chi)=\sum^{+\infty}_{n=1}a_n \chi^{\kappa(n)}=\sum_{\a\in \N^{\infty}_{\text{0,fin}}}\tilde a_{\a} \chi^{\a} ,\ \text{ where }\tilde a_{\a}=a_{p^{\a}}, \chi\in\D^{\infty}_{\text{fin}},
\end{equation*}
with the multiindex notation $\chi^{\a}=\chi_1^{\a_1}\chi_2^{\a_2}\cdots.$

Given a sequence of positive numbers $w=(w_n)_n=\left(w(n)\right)_n$, one considers the Hilbert space (see \cite{olsen,mcc})
$$\Hl^2_w:=\left\{\sum^{+\infty}_{n=1}a_n n^{-s}:\ \sum^{+\infty}_{n=1}\frac{\left|a_n\right|^2}{w_{n}}<+\infty\right\}.$$
The choice $w_n=1$ corresponds to the space  $\Hl^2$,  introduced in \cite{hakan3}.

The weights considered in this article  satisfy  $w_n=O(n^{\epsilon})$ for every $\epsilon>0$;  from the Cauchy-Schwarz inequality, Dirichlet series in $\Hw$ absolutely converge in $\C_{1/2}$. 

 We are interested in  the Volterra operator $T_g$ of symbol $g(s)=\sum^{+\infty}_{n=1} b_n n^{-s}$, defined by
\begin{equation}\label{Tg def}
T_g f(s):=-\int^{+\infty}_{s} f(w)g'(w)dw,\ \Re s>\frac{1}{2}. 
\end{equation}

On the unit disk $\D,$ the Volterra operator, whose symbol is an analytic function $g$, is given by
\begin{equation}\label{volterra D}
	J_gf(z):=\int^{z}_{0}f(u)g'(u)du,\  z\in \D.
\end{equation}
Pommerenke \cite{Pomme} showed that $J_g$ 
(\ref{volterra D}) 
is bounded on the Hardy space $H^2(\D)$ if and only if $g$ is in $BMOA(\D)$. Let $\s$ be the Haar measure on the unit circle $\T$.   Fefferman's duality Theorem states that $BMOA(\D)$ is the dual space of $H^1(\D)$. Thus the boundedness of  $J_g$  is equivalent to the boundedness of 
the Hankel form 
\begin{equation}\label{Hf D}
	H_g(f,h):=\int_{\T}f(u)h(u)\overline{g(u)}d\sigma(u),\ f,h\in H^2(\D).
\end{equation}

Let   $V$ be the  Lebesgue measure on $\C$, normalized such that $V(\D)=1$.

   Many authors, in particular \cite{ale-sisk}, have  studied Volterra operators on  Bergman spaces of $\D$.  The classical Bergman space $A^2_{\gamma}(\D)$, $\gamma>0$, is associated to the measure $d\tilde{m}_{\gamma}(z):=\gamma\left(1-\left|z\right|^2\right)^{\gamma-1}dV(z)$.  $J_g$ is bounded    on  $A^2_{\gamma}(\D)$  if and only if $g$ is in the Bloch space, which is the dual of $A^1_{\gamma}(\D)$. \\
	
	The Bergman space of the finite polydisk $A^2_{\gamma}(\D^d)$, $d\geq 1$,  corresponds to the measure
	$$  d\widetilde{\nu}_{\gamma}(z):=d\tilde{m}_{\gamma}(z_1)\times\cdots\times d\tilde{m}_{\gamma}(z_d).$$
	The boundedness of the Hankel form 
	\begin{equation}\label{Hf Dd}
	H_g(f,h):=\int_{\D^d}f(z)h(z)\overline{g(z)}d\widetilde{\nu}_{\gamma}(z),\ f,h\in A^2_{\gamma}(\D^d),
\end{equation}
	 is equivalent to the membership of $g$ to the Bloch space  (see \cite{constantin50}), defined by
		\begin{equation*}
		\text{Bloch}(\D^d):=\left\{f:\D^d\to \C \text{ holomorphic : }\max_{\kappa\in \Il_d}\sup_{z\in\D^d}\left|\partial^{\kappa}f\left(\kappa.z\right)\right|\left(1-\left|z\right|\right)^{\kappa}<+\infty\right\},
	\end{equation*}
	where $\Il_d$ denotes the set of multi-indices $\kappa=\left(\kappa_1,\cdots,  \kappa_d\right)$, with entries in $\left\{0,1\right\}$, 
	and 
	\begin{align*}
	z=\left(z_1,\cdots,z_d\right),\ 
	\partial^{\kappa}=\partial^{\kappa_1}_{z_1}\cdots  \partial^{\kappa_d}_{z_d},\  \left(1-\left|z\right|\right)^{\kappa}= \left(1-\left|z_1\right|\right)^{\kappa_1} \cdots  \left(1-\left|z_d\right|\right)^{\kappa_d}.
	\end{align*}

	 Recall that for $0<p<\infty$, the Hardy space of Dirichlet series  $\Hl^p$ is the space of Dirichlet series $f \in \Dl$ such that $\Bl f$ is in $H^p(\D^{\infty})$, endowed with the norm 
	$$  \left\|f\right\|_{\Hl^p}:=\left\|\Bl f\right\|_{H^p(\D^{\infty})}=\left(\int_{\T^{\infty}}\left|\Bl f(z)\right|^p d\sigma_{\infty}(z)\right)^{1/p},$$
	 $\sigma_{\infty}$ being the Haar measure of the infinite polytorus $\T^{\infty}$.
	
	The norm in the space $\Hl^{\infty}:=H^{\infty}(\C_0)\cap \Dl$ is 
	$$  \left\|f\right\|_{\Hl^{\infty}}=\sup_{s\in\C_0}\left|f(s)\right|.$$
	Let $H^{\infty}(\D^{\infty})$ be the space of series $F$ which are finitely bounded, i.e.
	$$ \left\|F\right\|_{H^{\infty}(\D^{\infty})}=\sup_{m\in\N_0,z\in\D^{\infty} }\left|A_m F(z)\right|<\infty. $$
	Via the Bohr isomorphism, we have \cite{cole,hakan3}
		\begin{equation}\label{H infini}
		\left\|f\right\|_{\Hl^{\infty}}=\left\|\Bl f\right\|_{H^{\infty}(\D^{\infty})}.
	\end{equation}
	
	Several abscissae are related to a function  $g$  in $\Dl$, of the form $g(s)=\sum^{+\infty}_{n=1}b_n n^{-s}$:
 \begin{align*}
\text{the abscissa of  convergence }&
   \sigma_c=\inf\left\{\sigma\in \R\ :\ \sum^{+\infty}_{n=1}b_n n^{-\sigma}\text{ converges }\right\};\\
\text{the abscissa of absolute convergence }&
   \sigma_a=\inf\left\{\sigma\in \R\ :\ \sum^{+\infty}_{n=1}\left|b_n \right|n^{-\sigma}\text{ converges }\right\};\\
\text{the  abscissa of uniform convergence }&
 \sigma_u=\inf\left\{\theta\in \R\ :\ \sum^{+\infty}_{n=1}b_n n^{-s}\text{ converges uniformly in }\C_{\theta}\right\}. 
\end{align*}
The abscissa of regularity and boundedness, denoted by $\sigma_b$, is the infimum of those $\theta$ such that $g(s)$ has a bounded analytic continuation, to the half-plane $\Re(s)>\theta+\epsilon$, for every $\epsilon>0.$ 

We have $-\infty\leq \sigma_c\leq \sigma_u\leq \sigma_a\leq+\infty$, and, if any of the abscissae is finite $\sigma_a-\sigma_c\leq 1$. Moreover, it is known that  $\sigma_b=\sigma_u$ \cite{bohr} , and $\sigma_a-\sigma_u\leq\frac{1}{2}$.
	
		Volterra operators (\ref{Tg def}) on the spaces $\Hl^p$ have been investigated in \cite{bre-perf-seip}. Our aim is to study similar questions  for the spaces $\Hw$,
	 associated to specific weights $w$ in the class $\Wl$ defined below.
\begin{defi}\label{mathcal W}
Let $\beta>0$. A sequence  $w$  belongs to  $\Wl$ if it has one of the following forms:
\begin{enumerate}\label{weight}
	\item $w_n=[d(n)]^{\beta}$,  where $d(n)$ is the number of divisors of the integer $n$. Then $\Hw:=\Ba$.
	\item $w_n=d_{\beta+1}(n)$,  where $d_{\gamma}(n)$ are the Dirichlet coefficients of the power of the Riemann zeta function, namely $ \zeta^{\gamma}(s)=\sum^{+\infty}_{n=1}d_{\gamma}(n)n^{-s} .$ Then  $\Hw:=\Ab$.
\end{enumerate} 
\end{defi}
As in the case of $\Hl^2$  \cite{bre-perf-seip}, we obtain sufficient / necessary conditions for $T_g$ to be bounded on the Hilbert spaces $\Hw$. However, due to the lack of information of the behavior of the symbols in the strip $0< \Re s  <1/2$, it seems difficult to get an " if and only if" condition. In the Hardy space setting, it is shown that  $T_g$ is bounded on $\Hl^2$ provided that $g$ in $BMOA({\Bbb C}_{0})$. Since the spaces $\Ab$ and $\Ba$ (see section \ref{spaces}) locally behave like Bergman spaces of the half plane $\C_0$, we would expect that the membership of $g$ in $\text{Bloch}({\Bbb C}_{0})$ (resp. $\text{Bloch}_0({\Bbb C}_{0})$) would imply the boundedness (resp. compactness) of $T_g$ on $\Hl^{2}_{w}.$ 
We obtain such a sufficient condition when $\Bl g$ depends on a finite number of variables $z_1,\cdots,z_d$. However, our method specfically uses that  $d$ is finite, and we do not know whether the same result holds if $\Bl g$  is a function of infinitely many variables.

 Le $\Nf_d$ be the set of positive integers which are multiples of the primes $p_1,\cdots,p_d$, 
 \begin{equation*}
	 \Dl_d:=\left\{f\in\Dl\ :\ f(s)=\sum_{n\in\Nf_d}a_n n^{-s}\right\},\ \text{ and }\Hl^p_{d,w}:=\Hpa\cap \Dl_d.  
 \end{equation*}

One of our main results is the following.

\begin{thm}\label{boundedness hilbert}
 Let  $T_g$ be the operator defined by (\ref{Tg def}) for some Dirichlet series $g$ in $\Dl$.

(a) If $g(s)= \sum^{+\infty}_{n=2}b_n n^{-s}$ is in $\Dl_{d}\cap \text{Bloch}({\Bbb C}_{0})$, then $T_g $ is bounded on $\Hl^2_w$ and 
$$ \left\|T_g\right\|_{\Ll(\Hl_w)}\lesssim \left\|g\right\|_{\text{Bloch}({\Bbb C}_{0})}.$$

(b) If $g$ is in $BMOA({\Bbb C}_{0})$, then $T_g$ is bounded on  $\Hl^2_w$ and 
$$ \left\|T_g\right\|_{\Ll(\Hl_w)}\lesssim \left\|g\right\|_{BMOA(\C_0)}. $$
 
(c)
 If $T_g$ is bounded on  $\Hl^2_w$, then $g$ is in $\text{Bloch}({\Bbb C}_{1/2})$ and
 $$ \left\|g\right\|_{\text{Bloch}({\Bbb C}_{1/2})} \lesssim  \left\|T_g\right\|_{\Ll(\Hl_w)}.$$
\end{thm}

Via the Bohr lift, $\Hw$ are  $L^2$-spaces of  functions on the polydisk  $\D^{\infty}$. Precisely, there exists a probability measure $\mu_w$ on  $\D^{\infty}$ such that
$$\left\|f\right\|^2_{\Hw}=\int_{\D^{\infty}} \left|\Bl f(z)\right|^2 d\mu_w(z). $$

Analogously to the spaces $\Hl^p$, we define the space $\Hl^p_{w}$, $0<p<\infty$ (see Section \ref{spaces}), as the closure of Dirichlet polynomials under the norm (quasi-norm if $0<p<1$)
$$ \left\|f\right\|_{\Hl^p_{w}}= \left\|\Bl f\right\|_{L^p\left(\D^{\infty},\mu_w \right)}.$$

 Let $\Xw=\Xl(\Hl^2_w)$ be  the space  of symbols $g$ giving rise to bounded operators $T_g$ on $\Hl^2_w$. Our study provides the following strict  inclusions:
 $$ BMOA({\Bbb C}_{0})\cap \Dl\subset_{\neq}\Xw\subset_{\neq} \cap_{0<p<\infty}\Hl^{p}_{w}.$$

We will also compare $\Xw$  with other spaces of Dirichlet series, in particular with the dual of $\Hl^{1}_{w}$, and the space of symbols $g$ generating a bounded Hankel form
$$ H_g(fh):=\left\langle fh,g\right\rangle_{\Hw} $$
on the weak product $\Hw\odot \Hw$. As in the case of $\Hl^2$  \cite{bre-perf-seip}, we only get partial results.\\
For Dirichlet series involving $d$ primes, we have 
$$ \Dl_{d}\cap \text{Bloch}(\C_0)\subset \Dl_{d}\cap \Xw \subset_{\neq}\Bl^{-1}\text{Bloch}(\D^d).$$ 

The paper is organized as follows. Section \ref{spaces} starts by presenting  some properties of  the spaces $\Hw$. As a space of analytic functions on the half-plane $\C_{1/2}$, $\Hw$ is continuously embedded
in a space of Bergman  type of $\C_{1/2}$. In view of  the Bohr lift, the norm of $\Hw$ can be expressed in terms of a probability measure  
  $\mu_{w} $  on the polydisk. For $0<p<\infty$, we consider the Bohr-bergman space $\Hl^p_{w}$,
	and derive equivalent norms for these spaces.
	
	In section \ref{spaces symbols}, we present some properties of the Dirichlet series which belong to a BMO or Bloch space of some half-plane $\C_{\theta}$. In particular, we relate the Carleson measures for both spaces of Dirichlet series and Bergman  type spaces.

Section \ref{bd} is devoted to the proof of Theorem \ref{boundedness hilbert}. First we consider the case when $g$ is a function of $p^{-s}_{1}, \cdots,p^{-s}_{d}$. To prove $(b)$, we observe that   the  boundedness of $T_g$ on $\Hl^2$ implies the boundedness of $T_g$ on $\Hl^2_w$. On another hand, combining the fact that $\Hw$ is embedded in a  Bergman  type space of the half-plane  $\C_{1/2}$ with some characterizations of Carleson measures, we establish that
$$\Xw\subset \text{Bloch}({\Bbb C}_{1/2}).$$
Compactness and Schatten  classes are considered in Sections \ref{compa} and \ref{scha}.	

In section \ref{ex}, we consider some specific symbols: fractional primitives of translates of a "weighted zeta"-function
and homogeneous symbols. These examples will be used in section  \ref{compar}.

In Section \ref{compar},  we investigate the relationship between the boundedness of the Volterra operator $T_g$, the boundedness of the Hankel form 
$$H_g(fh)=\left\langle fh,g\right\rangle_{\Hw},$$ and the membership of $g$ in the dual of $\Hl^1_w$. In particular, we study examples of Hankel forms on Bergman spaces of Dirichlet series, which are the counterparts  of the Hilbert multiplicative matrix \cite{bre-perf-seip-sis}.\\
Additionally, we show the strictness of  the inclusions derived previously
$$  BMOA(\C_0)\cap \Dl\subset_{\neq} \Xw\subset_{\neq} \cap_{0<p<\infty}\Hl^{p}_{w},$$  
 and compare  the space $\Dl_{d}\cap \Xw$ with Bloch spaces.

 For two functions $f, g$, the notation $f=O(g)$ or $f \lesssim g$,  means that there exists a constant $C$ such that $f\leq C g$ .
 If  $f=O(g)$ and $g=O(f)$, we  write $f\asymp g$.

\section{The Bohr-Bergman spaces $\Ba$, $\Ab$}\label{spaces}

\subsection{The spaces $\Ba$, $\Ab$}
These spaces are related to number theory. The number of divisors of the integer $n$,  $d(n)$, is $d(n)= (\kappa_1+1)\cdots(\kappa_d+1)$ when $n=p^{\kappa}$. We consider the following scale of Hilbert spaces
$$	\Bl^2_{\beta}=\left\{f(s)= \sum^{+\infty}_{n=1}a_n n^{-s}:\ \left\|f\right\|_{\Bl^2_{\beta}}:=\left(\sum^{n=1}_{+\infty}\frac{\left|a_n\right|^2}{\left[d(n)\right]^{\beta}}\right)^{\frac{1}{2}}<\infty\right\},\ \text{ for } \beta> 0 .$$
The case $\beta=0$ corresponds to the Hardy space $\Hl^2$. The reproducing kernels of $\Bl^2_{\beta}$ are $$K^{\Ba}(s,u)=\zeta_{\beta}(s+\overline{u}),\text{  where } 
	\zeta_{\beta}(s)=\sum_{n=1}^{+\infty}\left[d(n)\right]^{\beta}n^{-s}.$$
It is shown in \cite{wilson} that there exists $\phi_{\beta}(s)$,  an Euler product which converges absolutely in $\C_{1/2}$, such that 
	$$\zeta_{\beta}(s)=\left[\zeta(s)\right]^{2^{\beta}}\phi_{\beta}(s),\text{ and }\phi_{\beta}(1)\neq 0.$$

Another family of spaces  arises from the so-called generalized divisor function. For $\gamma>0$, the numbers $d_{\gamma}(n)$ are defined by the relation 
$$ \zeta^{\gamma}(s)=\sum^{+\infty}_{n=1}d_{\gamma}(n)n^{-s} .$$
A computation involving Euler products shows that we have
\begin{equation*}
	d_{\gamma}(p^r)=\frac{\gamma(\gamma+1)\cdots(\gamma+r-1) }{r!},\ \text{ for }p\in\PP,\ \text{ and any integer }r.
\end{equation*}
 
From its definition, $d_{\gamma}$ is a multiplicative function, i.e. $d_{\gamma}(kl)=d_{\gamma}(k)d_{\gamma}(l)$ if $k$ and $l$ are relatively prime. Thus, $d_{\gamma}(n)$ can be computed explicitly from the decomposition $n=p^{\kappa}$.

We  define the spaces 
$$	\Al^2_{\beta}=\left\{f(s)= \sum^{+\infty}_{n=1}a_n n^{-s}:\ \left\|f\right\|_{\Al^2_{\beta}}:=\left(\sum^{n=1}_{+\infty}\frac{\left|a_n\right|^2}{d_{\beta+1}(n)}\right)^{\frac{1}{2}}<\infty\right\},\ \text{ for }\beta>0,$$
with reproducing kernels  $K^{\Ab}(s,u)=\zeta^{\beta+1}(s+\overline{u})$.\\


Notice that, in each case, the reproducing kernel has the form  
$$K^{\Hw}(s,u)=Z_w(s+\overline u),$$
 where  $Z_w(s):=\sum^{+\infty}_{n=1}w_n n^{-s}$ has a singularity at $s=1$, with an estimate of the type
\begin{equation}\label{estim Zw}
	Z_w(s)=C_w(s-1)^{-\left(\delta+1\right)}\left[1+O(1)\right].
\end{equation}

\subsection{ Bohr-Bergman spaces on $\D^\infty$}
The Bohr correspondence is an isometry between  $\Hl^2_{w}$ and  the weighted Bergman space of the infinite polydisk
$$H^2_w(\D^{\infty})= \left\{\sum_{\nu\in\N^{\infty}_{0,\text{fin}}}a_{\nu} z^{\nu}:\ \sum_{\nu}\frac{\left|a_{\nu}\right|^2}{w_{\nu}}<\infty\right\},\ \text{ where }
w_{\nu}=\prod_{j}w_{\nu_j}.$$ 
 In particular, the space $\Hl^2$ is identified with the Hardy space $H^2(\T^{\infty})$ \cite{hakan3}.

 Let us consider the following probability measures on the unit disk $\D$,
	$$dm_{w}(z):=M(\left|z\right|^2)dV(z),\ \text{ where }M(r)=\begin{cases}\frac{1}{\Gamma(\beta)}\left(\log\frac{1}{r}\right)^{\beta-1},\ \text{ if }w_n=\left[d(n)\right]^{\beta} ,\\
	\beta (1-r)^{\beta-1},\ \text{ if } w_n= d_{\beta+1}(n)
	\end{cases}\ \beta>0.$$

 On the finite polydisk  $\D^d$ $(d\in\N)$, the corresponding  Bergman spaces $H^2_w(\D^d)$ - specifically $B^{2}_{\beta}(\D^d)$ and $A^{2}_{\beta}(\D^d)$-
are the  $L^2-$closures of polynomials with respect to the norm
$$ \left\|f\right\|_{H^2_w(\D^d)} :=\left(\int_{\D^d}\left|f(z_1,\cdots,z_d)\right|^2 dm_{w}(z_1)\times\cdots\times dm_{w}(z_d)\right)^{1/2}$$
If $f(z)=\sum_{n\in\N^d}a_n z^n$ is defined on $\D^d$, we have
\begin{equation}\label{norm B2 alpha}
	\left\|f\right\|^{2}_{B^{2}_{\beta}(\D)} =\sum_{n\in\N}\frac{\left|a_n\right|^2}{\left(n+1\right)^{\beta}} \text{ and  }\left\|f\right\|^{2}_{A^{2}_{\beta }(\D)} =\sum_{n\in\N} \left|a_n\right|^2\frac{n!}{(\beta +1)(\beta +2)\cdots(\beta +n)}.
\end{equation}


When $d$ is finite, the estimate 
$$\frac{n!}{(\beta +1)(\beta +2)\cdots(\beta +n)}\asymp(1+n)^{-\beta}  $$
yields that, 
  the spaces $B^{2}_{\beta}(\D^d)$ and $A^{2}_{\beta }(\D^d)$ coincide as sets, with equivalent norms. However, the norms are no longer equivalent in the case of infinitely many variables.

 The $\Hw$-norm  will be computed via the
rotation invariant  probability measure 
	$$ d\mu_{w}(\chi)=dm_{w}(\chi_1)\times dm_{w}(\chi_2) \times dm_{w}(\chi_3) \times\cdots\text{  on  }\D^{\infty}.$$
  Applying the Bohr lift to   a Dirichlet series  $f(s)= \sum^{+\infty}_{n=1}a_n n^{-s}$, and using (\ref{norm B2 alpha}) for each variable, one obtains the following formula  (see \cite{bai-bre} in the case of $\Ba$)
	$$ \int_{{\Bbb D}^{\infty}}\left|\Bl f(\chi)\right|^2 d\mu_{w}(\chi)= \sum^{+\infty}_{n=1}\frac{\left|a_n\right|^2}{w_n}=\left\|f\right\|^{2}_{\Hw}.$$

\begin{defi}
 For $0< p<\infty$, the Bohr-Bergman spaces of Dirichlet series  $\Bl^{p}_{\beta}$ and $\Al^{p}_{\beta}$ - denoted by $\Hl^{p}_{w}$ - are the completions of the Dirichlet polynomials in the norm (quasi norm when $0<p<1$)
$$ \left\|f\right\|^{p}_{\Hl^{p}_{w}}:= \int_{{\Bbb D}^{\infty}}\left|\Bl f(\chi)\right|^p d\mu_{w}(\chi).$$
\end{defi}
The Kronecker flow of the point $\chi=(\chi_1,\chi_2,\cdots)\in\Bbb C^{\infty}$ is given by 
$$ \Tl_t(\chi)=\left(2^{-it }\chi_1, 3^{-it }\chi_2,5^{-it }\chi_3,\cdots\right),\ t\in\Bbb R, $$
which defines an  ergodic flow on ${\Bbb T}^{\infty}$ by Kronecker's theorem. 

Therefore, it follows from Fubini's Theorem that, for any rotation invariant probability measure $d\nu$ on $\D^{\infty}$ and any probability measure $d\lambda$ on $\R$, we have
\begin{equation}\label{L2 d nu}
	\left\|f\right\|^{p}_{L^{p}\left(\D^{\infty},d\nu\right)}=\int_{\D^{\infty}}\int_{\R}\left|\left(\Bl f\right)(\Tl_t\chi)\right|^p d\lambda(t)d\nu\left(\chi\right).
\end{equation}

\subsection{On the half-plane $\C_{1/2}$}
For $\theta\in \R$, let $\tau_{\theta}$ be the following mapping from $\D$ to $\C_{\theta}$,
\begin{equation}\label{tau theta}
	\tau_{\theta}(z)=\theta+\frac{1+z}{1-z}.  
\end{equation}
For $ \delta>0$, the conformally invariant Bergman space $A_{i,\delta}\left(\C_{1/2}\right)$ is the space of those functions $f$ which are analytic in $\C_{1/2}$, and such that
$$	\left\|f\right\|^{2}_{A_{i,\delta}\left(\C_{1/2}\right)}:=\left\|f\circ \tau_{1/2} \right\|^{2}_{A^2_{\delta}\left(\D\right)}=4^{\delta}\delta\int_{\C_{1/2}}\left|f(s)\right|^2
	\frac{\left(\sigma-\frac{1}{2}\right)^{\delta-1}}{\left|s+\frac{1}{2}\right|^{2\delta+2}}dm(s)<\infty.$$


The weights $w$ of the class $\Wl$ satisfy a Chebyshev-type estimate 
\begin{equation}\label{def delta}
\sum_{n\leq x}w_n \asymp x\left(\log x\right)^{\delta},\ \text{ where }
\delta=\delta(w):=
\begin{cases}
2^{\b}-1 \text{ if }w_n=\left[d(n)\right]^{\b},\\
\beta\phantom{1111}\text{ if }w_n=d_{\beta+1}(n).
\end{cases}
\end{equation}

For any real number $\tau$, set $S_{\tau}=\left[\frac{1}{2},1\right]\times \left[\tau,\tau+1\right]$. As mentioned in the introduction, the Dirichlet series which belong the $\Hl^{2}_{w}$ absolutely converge in $\C_{1/2}$.
The space $\Hw$ is locally embedded in $A_{i,\delta(w)}\left(\C_{1/2}\right)$  \cite{olsen,{olsen-seip}}, which means
		$$\sup_{\tau\in\R}\int_{S_{\tau}}\left|f(s)\right|^2\frac{\left(\sigma-\frac{1}{2}\right)^{\delta-1}}{\left|s+\frac{1}{2}\right|^{2\delta+2}}dm(s)\leq c\left(\Hw\right)\left\|f\right\|^{2}_{\Hw}.$$

Since functions in $\Hl^2_w$ are uniformly bounded in $\C_1$, these embeddings are global (see \cite{bai-bre,bayart-bre}).

\begin{lem}\label{embedding}
Let 
 $\delta=\delta(w)$ be defined in (\ref{def delta}). 
Then $\Hw$ is continuously embedded in $A_{i,\delta}\left(\C_{1/2}\right)$. 
 \end{lem}

 

\subsection{Generalized vertical limits}
Every  $\chi=(\chi_1,\chi_2,\cdots)$ in $\Bbb C^{\infty}$ defines a completely multiplicative function by the formula $\chi(n)= \chi^{\kappa}$, where $n=p^{\kappa}$.  For $f$ of the form 
(\ref{dirichlet series def}),  the twisted Dirichlet series  \cite{bai-bre,bai-le}, is defined by
\begin{equation}\label{dirichlet series chi def}
	f_{\chi}(s)= \sum^{+\infty}_{n=1}a_n\chi(n) n^{-s}.
\end{equation}
Notice that if $\chi\in {\Bbb T}^{\infty}$, $f_{\chi}$ is the vertical limit of $f$, introduced in \cite{hakan3}. 

We also consider the translations $f_{\delta}(s)=f(s+\delta)$, $\d\in\R$.
For those $\chi\in {\Bbb D}^{\infty}$ and $s=\sigma+it$ for which the series (\ref{dirichlet series chi def}) converges, we have 
\begin{equation}\label{bohr f chi}
	f_{\chi}(s)= \left(\Bl f_{\sigma}\Tl_t\right)(\chi).
\end{equation}
When $f$ is in $\Hw$,   the Cauchy-Schwarz inequality implies  that (\ref{bohr f chi}) holds whenever $s\in \C_{1/2}$ and $\chi\in \overline{\D}^{\infty}$. By the Rademacher-Menchov Theorem (see \cite{olevs}),  (\ref{bohr f chi}) can be extended in the following way (the argument given in \cite{bai-bre} for $\Bl^2_{\beta}$ remains true for $\Al^2_{\beta}$).

\begin{lem}
If $f$ is in $\Hw$, the Dirichlet series $f_{\chi}$ as defined in (\ref{dirichlet series chi def}) converges in ${\Bbb C}_{0}$
for almost every $\chi\in {\Bbb D}^{\infty}$, with respect to $\mu_{w}$. 
\end{lem}

Recall that $\tau_{\theta}$,   $\theta\in\R$, is the conformal mapping defined in (\ref{tau theta}). For $0<  p<\infty$, the conformally invariant Hardy space $H^{p}_{i}\left(\C_{\theta}\right)$, is the space of those functions $f$ such that $f\circ \tau_{\theta}$ is in $H^p(\T)$, the usual Hardy space of the unit disk. Setting  $d\lambda(t)=\pi^{-1}(1+t^2)^{-1}dt$, we get
$$ \left\|f\right\|^p_{H^{p}_{i}\left(\C_{\theta}\right)}= \int_{\R}\left|f\left(\theta+it\right)\right|^{p}d\lambda(t)=\frac{1}{2\pi}\int^{\pi}_{-\pi}\left|f\circ \tau_{\theta}(u) \right|^p du,
\ \text{ for } f\in H^{p}_{i}\left(\C_{\theta}\right).$$

Let $f$ be in $\Hl^p_{w}$. In view of relation (\ref{L2 d nu}), and using the same argument as in \cite{bai-le,hakan3}, one can prove that for almost all $\chi$, with respect to $\mu_{w}$ , $f_{\chi}$ can be extended analytically on $\C_0$ to an element of $H^{p}_{i}\left(\C_{0}\right)$.The norm of $f$ in $\Hl^p_{w}$ can be expressed as
\begin{equation}\label{norm B f chi}
	\left\|f\right\|^{p}_{\Hl^p_{w}}=\int_{\D^{\infty}}\left\|f_{\chi}\right\|^{p}_{H^{p}_{i}\left(\C_{0}\right)}d\mu_{w}(\chi).
\end{equation}

\subsection{A Littlewood-Paley formula}
   We now derive  another expression for the norm in  $\Hl^p_{w}$.

	


\begin{prop}\label{PLf}
 Let $\lambda$ be a probability measure on $\R$, and   $p\geq 1$.\\
(a) If $f\in\Hpa$, then  $ \left\|f\right\|^{p}_{\Hpa}\asymp I_p(f)$, where
$$  I_p(f):=
\left|f(+\infty)\right|^p+4\int_{\D^{\infty}}\int_{\R}\int^{+\infty}_{0}\left|f_{\chi}( y+it)\right|^{p-2}\left|f'_{\chi}( y+it)\right|^2ydy d\lambda(t) d\mu_{w}(\chi).$$
When $p=2$, we have $\left\|f\right\|^{2}_{\Hw}= I_2(f)$.

(b) Let $f\in\Dl$, $f(s)=\sum^{+\infty}_{n=1}a_n n^{-s},$ such that   $f$ and  $f_{\chi}$   converge on $\C_0$  for a.a. $\chi\in\D^{\infty}$. If $I_p(f)<\infty$, then $f\in\Hpa$.

\end{prop}

\begin{proof}
Since the real variable $t$ corresponds to a rotation in each variable of $\D^{\infty}$, the rotation invariance of $\mu_w$ entails that $I_p(f)$ does not depend on the choice of the probability measure $\lambda$.
For general $p\geq 1$, we prove (a), by  using (\ref{norm B f chi}). We adapt the argument  from  \cite{bayart-queffelec}  (for $\Hl^p$),  by integrating over the polydisk ${\Bbb D}^{\infty}$ instead of the polytorus ${\Bbb T}^{\infty}$.


Suppose  $f$ is in $\Hw$, and take $y>0$.  From (\ref{L2 d nu}) and the rotation invariance, we obtain
\begin{align*}
\int_{\R}\int_{\D^\infty}\left|f'_{ \chi}(y+it)\right|^2 d\mu_{w}( \chi) d\lambda(t)
=\int_{\D^\infty}\left|\Bl f'_{ y}(\chi)\right|^2 d\mu_{w}( \chi)
=\sum^{+\infty}_{n=1}\frac{\left|a_n\right|^2}{w_n}(\log n)^2 n^{-2y}.
\end{align*}
Integration against $y$ on $(0,+\infty)$ gives the formula (see details in
\cite{bayart2} for the case of $\Hl^2$).

If $f$ is as in (b), the integrand in  $I_p(f)$ is measurable. For  $\chi\in\D^{\infty}$, the change of variables $s=y+it=\omega(z)=2\frac{1+z}{1-z}$ transfers the Littlewood-Paley formula 
from $\D$ to
$\C_0$,
\begin{align*}
 \int_{\R}\left|f_{ \chi}(it)\right|^p \frac{2}{\pi(2^2+t^2)}dt&\asymp   \left|f_{ \chi}(2)\right|^p+\int_{\D}\left(1-\left|z\right|^2\right)  \left|f_{ \chi}(\omega(z))\right|^{p-2}\left|f'_{ \chi}(\omega(z))\right|^2\left|\omega'(z)\right|^2 dV(z) \\
&\asymp  \left|f_{ \chi}(2)\right|^p+\int^{+\infty}_{0}\int_{\R}\frac{2y}{(y+2)^2+t^2}\left|f_{ \chi}(y+it)\right|^{p-2}\left|f'_{ \chi}(y+it)\right|^2dtdy\\
&\lesssim\left\|f^*\right\|^{p}_{L^{\infty}(\overline{\C_2})}+\int^{+\infty}_{0}\int_{\R}\frac{y}{1+t^2}\left|f_{ \chi}(y+it)\right|^{p-2}\left|f'_{ \chi}(y+it)\right|^2dtdy,
\end{align*}
where $f^*(s):=\sum^{+\infty}_{n=1}\left|a_n\right|n^{-s}$ is bounded on $\overline{\C_2}$. 

Integrating on $\D^\infty$ with respect to $\mu_w$, and using (\ref{L2 d nu}), we get that
$$\left\|\Bl f\right\|^{p}_{L^p(\D^\infty, \mu_w)} \lesssim\left\|f^*\right\|^{p}_{L^{\infty}(\overline{\C_2})}+I_p(f)<\infty. $$
Therefore, $\Bl f\in L^p(\D^\infty, \mu_w)$. The  martingale $(A_m \Bl f)_m$ (with respect to the increasing sequence of $\s$-algebras of the sets $\D^m\times\left\{0\right\}$) converges in $L^p(\D^\infty, \mu_w)$ to $\Bl f$. Polynomial approximation in the Bergman spaces of the finite polydisks $\D^m$ shows that $\Bl f$ is in $\Bl \Hpa$.

\end{proof}


\section{Spaces of symbols of Volterra operators in half-planes}\label{spaces symbols}
If $g$ is in $\Dl$, the definition  (\ref{Tg def}) of  $T_g$ 
  shows that we can assume that $g\left(+\infty\right)=0$, i.e.
	\begin{align*}
g(s)=\sum^{+\infty}_{n=2}b_n n^{-s}.  
\end{align*}
As  in the study of Volterra operators on Bergman spaces the unit disk \cite{ale-sisk},  and on the space of Dirichlet series $\Hl^2$ \cite{bre-perf-seip},  the boundedness of $T_g$ on $\Hw$ will be related to   Carleson measures, and to the membership of $g$ to a BMO space or a  Bloch space.

Let $Y$ be either $\Hw$ or the Bergman space $A_{i,\delta}\left(\C_{1/2}\right)$,  $\delta>0$. A positive Borel measure $\mu$ on $\C_{1/2}$ is called a Carleson measure for $Y$ if there exists a constant $C$ such that,  
$$\int_{ \C_{1/2} }
\left|f\right|^2 d\mu\leq C\left\|f\right\|^2_{Y}\text{ for all }f\in Y .$$
The smallest such constant, denoted by $\left\|\mu\right\|_{CM(Y)}$, is called the Carleson constant for $\mu$ with respect to $Y$. A Carleson measure $\mu$ is a vanishing Carleson measure for $Y$ if we have 
$$  \lim_{k\to \infty}\int_{ \C_{1/2} }
\left|f_k\right|^2 d\mu=0,$$
for every weakly compact sequence $(f_k)_k$ in $Y$ (which means that $\left\|f_k\right\|_{Y}$ is bounded and $f_k(s)\rightarrow 0$ on every compact set of $\C_{1/2}$).

\subsection{$BMO$ spaces of Dirichlet series}
The space $BMOA(\C_{\theta})$ consists of holomorphic functions $g$ in the half-plane $\C_{\theta}$ which satisfy
$$\left\|g\right\|_{BMO(\C_{\theta})}:=\sup_{I\subset \R}\frac{1}{\left|I\right|}\int_{I}\left|g(\theta+it)-\frac{1}{\left|I\right|}\int_{I}g(\theta+i\tau)d\tau\right|dt<\infty.$$

Any  $g$  in $\Dl\cap BMOA(\C_{0})$ has an abscissa of boundedness $\sigma_b\leq 0$ (Lemma 2.1 of \cite{bre-perf-seip}).

The space $VMOA(\C_0)$ consists in those functions $g$ in $BMOA(\C_0)$ such that 
$$ \lim_{\delta\rightarrow 0^+}\sup_{\left|I\right|<\delta}\frac{1}{\left|I\right|}\int_{I}\left|f(it) - \frac{1}{\left|I\right|}\int_I f(i\tau) d\tau\right| dt=0.$$

\subsection{Bloch spaces of Dirichlet series}
The Bloch space $\text{Bloch}({\Bbb C}_{\theta})$ consists of holomorphic functions in the half-plane ${\Bbb C}_{\theta}$ which satisfy
$$  \left\|g\right\|_{\text{Bloch}({\Bbb C}_{\theta})}:=\sup_{\sigma+it\in {\Bbb C}_{\theta}}\left(\sigma-\theta\right)\left|f'(\sigma+it)\right|.$$

\begin{lem}\label{D inter Bloch}
If $g$ be in $\Dl\cap \text{Bloch}({\Bbb C}_{0})$.

 (a) Its abscissa of boundedness satifies $\sigma_b\leq 0$. 

(b) For every  $\chi\in\D^\infty$, $g_{\chi}$ is in $\text{Bloch}(\C_0)$, and $\left\|g_{\chi}\right\|_{\text{Bloch}(\C_0)}\leq\left\|g\right\|_{\text{Bloch}(\C_0)}$.

(c) Suppose that $ y_0>\frac{1}{2}$. Then  there exists a constant $C=C(y_0)$, such that,  
$$  \left|g'_{\chi}(y+it)\right|\leq C 2^{-y} \left\|g\right\|_{\text{Bloch}(\C_0)} ,\ \text{ for all  }\chi\in\D^\infty,\  t\in\R, \ y\geq y_0. $$
\end{lem}

\begin{proof}
Let $\epsilon>0$. If  $s=\sigma+it$ is in $\C_0$, the definition of the Bloch-norm implies that
$$ \epsilon\left|g'(\epsilon+s)\right| \leq(\epsilon+\sigma)\left|g'(\epsilon+s)\right|\leq \left\|g\right\|_{\text{Bloch}({\Bbb C}_{0})}.$$
It follows that $g'$, and then    $g$ is bounded in $\C_{\epsilon}$; (a) is proved.

Now fix $\sigma>0$. 
  Let $m\geq 1$ be an integer, and $z=(z_1,\cdots,z_m,z_{m+1},\cdots),$ $\chi$ in $\D^\infty$. From the properties of  $\Hl^\infty$ and  the proof of (a),  we have
\begin{align*}
\left|A_m \Bl (g'_{\sigma})_{\chi}(z)\right|
=\left|A_m \Bl  g'_{\sigma}(z.{\chi})\right|
\leq\left\|\Bl  g'_{\sigma}\right\|_{H^\infty(\T^\infty)}=\left\|  g'_{\sigma}\right\|_{\Hl^\infty},
\end{align*}
and $ \left\| (g'_{\sigma})_{\chi}\right\|_{\Hl^\infty}=\left\|\Bl (g'_{\sigma})_{\chi}\right\|_{H^\infty(\T^\infty)}\leq\left\| g'_{\sigma}\right\|_{\Hl^\infty}$.
Therefore, $(g'_\s)_{\chi}$ is in $\Hl^\infty$;
 (b) holds, due to
$$ \sigma\left|g'_{\chi}(\sigma+it)\right| \leq\left\|g\right\|_{\text{Bloch}({\Bbb C}_{0})},\ \text{ for all }t\in \R, \chi\in\T^\infty,\sigma>0.$$

If $0<\delta<y_0-\frac{1}{2}$, the Cauchy-Schwarz inequality and Parseval's relation induce that
\begin{align*}
\left|g'_{\chi}(y+it)\right|^2&\leq\left(\sum^{+\infty}_{n=2}\left|b_n\right|(\log n) n^{-y}\right)^2
= \left(\sum^{+\infty}_{n=2}\left|b_n\right|(\log n) n^{-\frac{\delta}{2}} n^{-\left(\frac{\delta}{2}+\frac{1}{2}\right)}n^{-\left(y-\frac{1}{2}-\delta\right)}\right)^2\\
&\lesssim \zeta(1+\delta)2^{-2y}\left\|\Bl g'_{\delta/2}\right\|^{2}_{H^{2}\left(\T^\infty\right)}.
\end{align*}
We now get (c) from the chain of inequalities
$$  \left\|\Bl g'_{\delta/2}\right\|_{H^{2}\left(\T^\infty\right)}\leq\left\|\Bl g'_{\delta/2}\right\|_{H^{\infty}\left(\T^\infty\right)}=\left\| g'_{\delta/2}\right\|_{\Hl^{\infty}}\leq\frac{2}{\delta}\left\| g\right\|_{\text{Bloch}({\Bbb C}_{0})},$$

\end{proof}


Now, recall severals of Bloch functions, which   are extracted from \cite{ale-sisk,carleson}.

\begin{lem}\label{charact Bloch CM}
Assume $\delta>0$. For $g$ holomorphic  in ${\Bbb C}_{\theta}$, the following are equivalent:
\begin{enumerate}
	\item [(a)] $g\in\text{Bloch}({\Bbb C}_{\theta})$;
		\item [(b)] $h=g\circ \tau_{\theta}\in\text{Bloch}(\D)$;
		\item [(c)] The measure $d\mu_{{\Bbb C}_{\theta}, g}(s)=\left|g'(\sigma+it)\right|^2\frac{\left(\sigma-\theta\right)^{\delta+1}}{\left|s-\theta+1\right|^{2\delta+2}}d\sigma dt$ is a Carleson measure for $A_{i,\delta}({\Bbb C}_{\theta})$;
		\item [(d)] The measure $d\mu_{\D,h}(z)=\left|h'(z)\right|^2 \left(1-\left|z\right|^2\right)^{\delta+1}dm_1(z)$ is a Carleson measure for $A^2_{\delta}(\D)$;
		\item [(e)]  The operator $J_h$, given by
		$$  J_hf(z)=\int^{z}_{0}f(t) h'(t)dt,$$
		is bounded on $A^2_{\delta}(\D)$.
		\end{enumerate}
		Moreover, the quantities 
		$$ \left\|g\right\|_{\text{Bloch}({\Bbb C}_{\theta})}, \ \left\|\mu_{{\Bbb C}_{\theta}, g}\right\|_{CM(\C_{\theta})}, \left\|J_g\right\|_{\Ll\left(A^2_{\delta}(\D)\right)}$$
		are comparable.
\end{lem}

The little Bloch space is the space 
$$\text{Bloch}_0({\Bbb C}_{\theta})=\left\{f\in  \text{Bloch}({\Bbb C}_{\theta})\ : \ \lim_{\sigma\rightarrow \theta}\left(\sigma-\theta\right)\left|g'(s)\right| =0\right\}.$$
The membership in $\text{Bloch}_0({\Bbb C}_{\theta})$ is characterized by a little oh version of Lemma \ref{charact Bloch CM}, involving vanishing Carleson measures.

We show that Dirichlet polynomials are dense in $\Dl\cap \text{Bloch}_0({\Bbb C}_{0})$. For $g(s)=\sum_{n\geq 1}b_n n^{-s}$, the partial sum operator is defined by $S_N g(s)=\sum_{n= 1}^{N}b_n n^{-s}$.

\begin{prop}\label{density bloch0}
Let $g$ be in $\text{Bloch}_0({\Bbb C}_{0})\cap \Dl$, and $\epsilon>0$. Then there exists $P$ in $\Pl$ such that
$$  \left\|g-P\right\|_{\text{Bloch}({\Bbb C}_{0})}\leq\epsilon.$$
If in addition $g$ is in $\Dl_d$, $P$ can be chosen in $\Dl_d$.
\end{prop}

\begin{proof}
For every $\delta>0$, $g_\delta=g(\delta+.)$ is also  in $\text{Bloch}_0({\Bbb C}_{0})$. As $\d$ tends to $0$,  $(g_\delta)_\delta$ converges to $g$ uniformly on compact sets of $\C_0$, and 
$  \lim_{\sigma\to 0^+}\sigma\left|g'_\delta(s)\right|=0,$
uniformly with respect to $\delta\in(0,1)$. It then follows from \cite{acp} that 
$ \lim_{\delta\to 0^+}\left\|g-g_\delta\right\|_{\text{Bloch}({\Bbb C}_{0})}=0. $
Thus, we can choose $\delta>0$ such that 
$ \left\|g-g_\delta\right\|_{\text{Bloch}({\Bbb C}_{0}) } \leq\frac{\epsilon}{2}.$
Since $\sigma_b(g)=\sigma_u(g)\leq 0$, the partial sums $\left(S_N g\right)_N$ converge uniformly to $g$ in 
 $\overline {\C_\delta}$,
$ \lim_{N\to +\infty}\left\|S_N g_\delta-g_\delta\right\|_{\Hl^{\infty}}=0.$
For large $N$, the triangle inequality implies that
\begin{align*}
\left\|g-S_N g_\delta\right\|_{\text{Bloch}({\Bbb C}_{0})}&\leq\left\|g-g_\delta\right\|_{\text{Bloch}({\Bbb C}_{0}) }+\left\|g_\delta-S_N g_\delta\right\|_{\text{Bloch}({\Bbb C}_{0})}\\
&\leq \frac{\epsilon}{2}+2 \left\|S_N g_\delta-g_\delta\right\|_{\Hl^{\infty}}\leq\epsilon.
\end{align*}

\end{proof}
\subsection{Carleson measures on the half-plane $\C_{1/2}$}
 On 
 $\C_{1/2}$, we  consider Carleson squares 
$$  Q(s_0)=\left(\frac{1}{2}, \sigma_{0}\right]\times \left[t_0-\frac{\epsilon}{2},t_0+\frac{\epsilon}{2}\right], \text{ where }s_0=\sigma_{0}+it_0\in \C_{1/2}$$
 is the midpoint of the right edge of the square and $\epsilon=\sigma_{0}-\frac{1}{2}.$

We need the following property (see section 7.2 in \cite{zhu-b}).

\begin{lem}\label{CM bergman}
Let $\delta>0$ and let $\mu$ be a Borel measure on $\C_{1/2}$. Then $\mu$ is a Carleson measure for $ A_{i,\delta}\left(\C_{1/2}\right)$ if and only if, for every square  $ Q(s_0)$, with $s_0=\sigma_{0}+it_0$, we have
$$  \mu\left( Q(s_0)\right)=O\left(\left(2\sigma_{0}-1\right)^{\delta+1}\right)\text{ as }\sigma_0\to\left(\frac{1}{2}\right)^+.$$
In addition, $\mu$ is a vanishing Carleson measure for $ A_{i,\delta}\left(\C_{1/2}\right)$ if and only if, uniformly for $t_0$ in $\R$, 
$$  \mu\left( Q(s_0)\right)=o\left(\left(2\sigma_{0}-1\right)^{\delta+1}\right)\text{ as }\sigma_0\to\left(\frac{1}{2}\right)^+.$$

\end{lem}

By Lemma \ref{embedding},  $\Hw$ is embedded in the Bergman-type  space  $A_{i,\delta}\left(\C_{1/2}\right)$, the exponent $\delta=\delta(w)$ being defined in (\ref{def delta}). Bounded Carleson measures for both spaces $\Hw$ and $A_{i,\delta}\left(\C_{1/2}\right)$ have been compared in  \cite{olsen,olsen-saks,bayart-bre}. We extend their results.

\begin{lem}\label{CM B alpha A beta}
Let $\mu$ be a positive Borel measure on $\C_{1/2}$.
\begin{enumerate}
	\item 
	If $\mu$ is a Carleson measure (resp. vanishing Carleson measure) for $\Hw$, then $\mu$ is a Carleson measure (resp. vanishing Carleson measure) for $A_{i,\delta}\left(\C_{1/2}\right)$ and 
$$\left\|\mu\right\|_{CM\left(A_{i,\delta}\left(\C_{1/2}\right)\right)}\lesssim \left\|\mu\right\|_{CM\left(\Hw\right)}. $$
\item	 Assume that $\mu$  has bounded support. If $\mu$ is a Carleson measure (resp. vanishing Carleson measure) for $A_{i,\delta}\left(\C_{1/2}\right)$, then  $\mu$ is a Carleson measure (resp. vanishing Carleson measure) for $\Hw$ and 
$$\left\|\mu\right\|_{CM\left(\Hw\right)}\lesssim \left\|\mu\right\|_{CM\left(A_{i,\delta}\left(\C_{1/2}\right)\right)}. $$
\end{enumerate}
\end{lem}

\begin{proof}
Suppose that $\mu$ is a Carleson measure for $\Hw$, and  let $Q(s_0)$ be a small Carleson square in $\C_{1/2}$. For the test function $f_{s_0}(s)=K^{\Hw}(s,s_0)$, we have
\begin{align*}
\int_{Q(s_0)}\left|f_{s_0}\right|^2 d\mu\leq\int_{\C_{1/2}}\left|f_{s_0}\right|^2 d\mu\leq C(\mu)\left\|K^{\Hw}(.,s_0)\right\|^{2}_{\Hw}\le Z_w\left(\Re s_0\right).
\end{align*}
From the estimate
 of $Z_w$ (\ref{estim Zw}) and Lemma \ref{CM bergman}, $\mu$ is a Carleson measure for $A_{i,\delta}\left(\C_{1/2}\right)$, since
$$ \left(\Re s_0-\frac{1}{2}\right)^{-2\left(\delta+1\right)}\mu\left(Q(s_0)\right)\lesssim \left(\Re s_0-\frac{1}{2}\right)^{-\left(\delta+1\right)}. $$
For $\mu$  a Carleson measure for $A_{i,\delta}\left(\C_{1/2}\right)$ with  bounded support, (2) 
holds \cite{olsen,olsen-saks}.

As for  vanishing Carleson measures, the reasoning used in   \cite{bayart-bre} for $\Ba$ can be transfered to the spaces 
$\Ab$, with the test functions
$$  f_k(s)=\frac{K^{\Hw}(s,s_k)}{\left\|K^{\Hw}(.,s_k)\right\|_{\Hw}},$$ 
where $s_k=1/2+\epsilon_k+i \tau_k$ is a sequence in $\C_{1/2}$ such that $\epsilon_k\rightarrow 0$.
\end{proof}

We also require an equivalent norm for $A_{i,\delta}\left(\C_{1/2}\right)$, when $\delta>0$. For Bergman spaces of the unit disk, recall  the following consequence of Stanton's formula \cite{st, smith}:  
$$ \left\|h\right\|^{2}_{A_{\delta}\left(\D\right)}\asymp \left|h(0)\right|^2+\int_{\D}\left|h'(z)\right|^2 \left(1-\left|z\right|^2\right)^{\delta+1}dV(z),\ \text{ for }h \text{ holomorphic on }\D.$$ 
Via the 
mapping $\tau_{1/2}$, we obtain  that, for any $f$ holomorphic on $\C_{1/2}$, 
\begin{equation}\label{stanton}
	\left\|f\right\|^{2}_{A_{i,\delta}\left(\C_{1/2}\right)}\asymp \left|f(\frac{3}{2})\right|^2+\int_{\C_{1/2}}\left|f'(s)\right|^2 
	\frac{\left(\sigma-\frac{1}{2}\right)^{\delta+1}}{\left|s+\frac{1}{2}\right|^{2\delta+2}}dV(s).
\end{equation}
 
\section{Boundedness of $T_g$}\label{bd}
In this section, we   characterize functions in $\Xw$, and  prove Theorem \ref{boundedness hilbert}.

\subsection{Carleson measure characterization}
The boundedness of $T_g$ on  $\Hw$ can be described in terms of Carleson measures. This generalizes the setting of the Hardy space $\Hl^2$ \cite{bre-perf-seip}.

Recall that $\Hw$ is associated to the probability measure $ \mu_{w}$ on the polydisk $\D^{\infty}$. 

\begin{prop}\label{norm Tg via LPf}
$T_g$ is bounded on  $\Hw$  if and only if there exists a constant $C=C(g)$ such that 
\begin{equation}\label{CM Tg Y}
	\left\|T_g f\right\|^2_{\Hw}\asymp\int_{\D^{\infty}}\int_{\R}\int^{+\infty}_{0} \left|f_{\chi}(\sigma+it)\right|^2\left|g'_{\chi}(\sigma+it)\right|^2 \frac{\sigma d\sigma dt}{1+t^2}d\mu_w(\chi)\leq C^2 \left\|f\right\|^{2}_{\Hw},
\end{equation}
or, equivalently
\begin{equation}\label{CM Tg without t}
\int_{\D^{\infty}}\int^{+\infty}_{0} \left|f_{\chi}(\sigma)\right|^2\left|g'_{\chi}(\sigma)\right|^2 \sigma d\sigma d\mu_w(\chi)\leq C^2 \left\|f\right\|^{2}_{\Hw}.
\end{equation}
The smallest constant $C$ satisfying (\ref{CM Tg Y}) is such that $C\asymp \left\|T_g\right\|_{\Ll(\Hw)}$.
\end{prop}

\begin{proof}
 Applying  the Littlewood-Paley formula (Proposition  \ref{PLf})  to the measure $d\lambda(t)=\pi^{-1}(1+t^2)^{-1}dt$ and  the function $T_g f$, we get  (\ref{CM Tg Y}).

The rotation invariance of the measure $d\mu_w(\chi)$ gives 
(\ref{CM Tg without t} ).

\end{proof}

\subsection{Proof of Theorem \ref{boundedness hilbert} (a):  $\Bl g$ depends on a finite number of variables}
For $1\leq q$ and $d\geq 1,$ recall that
$f\in \Hl^q_{d,w}$ if and only if $f$ is in $\Hl^q_{w}$ and $\Bl f$ is a function of $z_1,\cdots,z_d$. 

When needed, we shall identify $z=(z_1,\cdots,z_d)\in\D^d$ with $(z,0)\in\D^d\times \left\{0\right\}$.

If $g(s)= \sum^{+\infty}_{n=2}b_n n^{-s}$ is in $\Hl^2_{d,w}$, we observe that for $z\in \D^d,$
$$\Bl g'(z) =\sum^{d}_{j=1}\log p_j \sum_{\alpha\in \N^d}\tilde b_{\alpha}\alpha_j z^{\alpha} =R \Bl g(z),$$
where $R$ is the operator $$  RG(z_1,\cdots,z_d)=\sum^{d}_{j=1}(\log p_j) z_j \partial _j G(z_1,\cdots,z_d). $$
 We define the set
$$  \Delta_{\epsilon}:=\left\{z=(z_1,\cdots,z_d)\in \D^d,\ \forall j,\ \left|z_j\right|< p^{-\epsilon}_{j}\right\},\ \text{ for }\epsilon>0.$$

Take $x>0$, $t\in\R$, and $z\in\D^d$. By construction, $z\in \overline{\Delta_{\sigma(z)}}$ and $  \sigma(\bpx.z)\geq\sigma(z)+x\frac{\log p_1}{\log p_d}$. 

For $g\in \Dl_{d}$, we write
$  g_z(x)=g_{(z,0)}(x)=\Bl g_x(z). $ Since $g$ is in $\text{Bloch}({\Bbb C}_0)$, we apply (\ref{H infini})  to $g'_x$, and get
\begin{equation}\label{ineq g'z(x)}
	\left|g'_z(x+it)\right|=\left|\Bl g'_x(\Tl_t z)\right|\leq \sup_{\zeta\in \overline{\Delta_{\sigma(\bpx.z)}}} \left|\Bl g'(\zeta)\right|= \sup_{s\in \overline{\C_{\sigma(\bpx.z)}}} \left|g'(s)\right| \leq  \frac{\log p_d}{\log p_1}\frac{\left\|g\right\|_{\text{Bloch}({\Bbb C}_0)}}{x+\sigma(z)},
\end{equation}


\begin{proof}[Proof of Theorem \ref{boundedness hilbert} (a)]
Let $f(s)=\sum_{n\geq 1}a_n n^{-s}$ be in $\Hl^2_{w}$, and, for $\chi=(z,z')\in\D^d\times \D^{\infty},$ 
$$  \Bl f(\chi)
=\sum_{(\alpha,\alpha')\in\N^d\times\N^{\infty}_{\text{0,fin}}}c_{\alpha,\alpha'}z^{\alpha}z'^{\alpha'}=\sum_{\alpha\in\N^d}c'_{\alpha}(z')z^{\alpha},
\text{ where }c'_{\alpha}(z')=\sum_{\alpha'\in\N^{\infty}_{\text{0,fin}}}c_{\alpha,\alpha'}z'^{\alpha'}.$$

In view of  Proposition \ref{norm Tg via LPf}, we  aim to estimate
$ \left\|T_g f\right\|^{2}_{\Hl^2_{w}}\asymp \Il_1+\Il_2, $ where  
\begin{align*}
 \Il_1&:= \int_{\D^{\infty}}\int^{1}_{0} \left|f_{\chi}(x)\right|^2\left|g'_{\chi}(x)\right|^2 x dx d\mu_w(\chi),\text{ and }
\Il_2:= \int_{\D^{\infty}}\int^{+\infty}_{1} \left|f_{\chi}(x)\right|^2\left|g'_{\chi}(x)\right|^2 x dx d\mu_w(\chi).
\end{align*}

By (\ref{ineq g'z(x)}), the rotation invariance and Fubini's Theorem, we have
\begin{align*}
 \Il_1
&\lesssim \left\|g\right\|_{\text{Bloch}({\Bbb C}_0)}^2 \int^{1}_{0}x\int_{\D^{\infty}}\int_{\D^d}\frac{1}{\left[x+\sigma(z)\right]^2}\left|\sum_{\alpha\in\N^d}c'_{\alpha}(\bppx. z')\left(z_1 p^{-x}_{1}\right)^{\alpha_1}\cdots\left(z_d p^{-x}_{d}\right)^{\alpha_d}\right|^2 d\mu_w(z,z')dx\\
&\lesssim \left\|g\right\|_{\text{Bloch}({\Bbb C}_0)}^2 \int_{\D^{\infty}}  \int^{1}_{0}x\sum_{\alpha\in\N^d}   \left|c'_{\alpha}(\bppx. z')\right|^2 I_\alpha (x) dx d\mu_w(z'),
\end{align*}
where 
$$  I_\alpha (x) :=\int_{\D^d}\frac{1}{\left[x+\sigma(z)\right]^2}\left|z_1 p^{-x}_{1}\right|^{2\alpha_1}\cdots\left|z_d p^{-x}_{d}\right|^{2\alpha_d}d\mu_w(z).$$
Using  the rotation invariance again as well as  the fact that $p_j\geq 1$, and setting  $\Jl_{\alpha}:=\int^{1}_{0}x I_\alpha (x) dx$, we get
\begin{align*}
 \Il_1&  \lesssim  \left\|g\right\|_{\text{Bloch}({\Bbb C}_0)}^2 \sum_{\alpha\in\N^d}  \int^{1}_{0}x I_\alpha (x) \left( \int_{\D^{\infty}}\left|\sum_{\alpha'} c_{\alpha,\alpha'}(\bppx. z')^{\alpha'}\right|^2 d\mu_w(z') \right)dx\\
&\lesssim  \left\|g\right\|_{\text{Bloch}({\Bbb C}_0)}^2 \sum_{\alpha,\alpha'}\left|c_{\alpha,\alpha'}\right|^2  \Jl_{\alpha}\left(\int_{\D^{\infty}}\left|z'^{\alpha'}\right|^2  d\mu_w(z') \right)
\lesssim  \left\|g\right\|_{\text{Bloch}({\Bbb C}_0)}^2 \sum_{\alpha,\alpha'}\frac{\left|c_{\alpha,\alpha'}\right|^2  \Jl_{\alpha}}{w\left(p^{ \alpha_{d+1}}_{d+1}\right)\cdots w\left(p^{ \alpha_{r}}_{r}\right)} .
\end{align*}
For the moment, we admit that
$ \Jl_{\alpha}\leq C(d,w)\left[\prod_{j=1}^{d}{w(p^{\alpha_j}_{j})}\right]^{-1},  $
which will be proved in Lemma \ref{int I alpha}. Hence,
$$ \Il_1\lesssim  \left\|g\right\|_{\text{Bloch}({\Bbb C}_0)}^2 \sum_{\alpha,\alpha'}
\frac{\left|c_{\alpha,\alpha'}\right|^2}{w(p^{(\alpha,\alpha')})}
\lesssim  \left\|g\right\|_{\text{Bloch}({\Bbb C}_0)}^2  \left\|f\right\|^{2}_{\Hl^{2}_{w}}.$$

Combining Lemma \ref{D inter Bloch} with the following observation, 
$$\int_{\D^{\infty}}\left| f_{\chi}(x)\right|^2 d\mu_w(\chi)= \int_{\D^{\infty}}\left|\sum_{n=p^{\alpha}}a_n n^{-x}\chi^{\alpha}\right|^2 d\mu_w(\chi)
=\sum_{n\geq 1}\frac{\left|a_n\right|^2  n^{-2x}}{w_n}\leq \left\|f\right\|^{2}_{\Hl^{2}_{w}},$$
 we estimate $\Il_2$,
$$\Il_2 \lesssim\int^{+\infty}_{1}x \int_{\D^{\infty}}\left\|g\right\|^{2}_{\text{Bloch}(\C_0)}4^{-x} \left| f_{\chi}(x)\right|^2 d\mu_w(\chi)dx \lesssim  \left\|g\right\|^{2}|_{\text{Bloch}({\Bbb C}_0)}\left\|f\right\|^{2}_{\Hl^{2}_{w}}.$$
\end{proof}

 Recall that
$$ I_\alpha(x)=\int_{\D^d}\frac{1}{\left[x+\sigma(z)\right]^2}\left|z_1 p^{-x}_{1}\right|^{2\alpha_1}\cdots\left|z_d p^{-x}_{d}\right|^{2\alpha_d}d\mu_w(z),\ \alpha\in\N^d,\ 0<x<1. $$

\begin{lem}\label{int I alpha}
There exists a constant $C=C(w,d)$, such that
$$ \Jl_{\alpha}:= \int^{1}_{0}xI_\alpha(x)dx\leq C \prod^{d}_{j=1} \frac{1}{w\left(p^{\alpha_j}_{j}\right)}. $$
\end{lem}
The proof of Lemma \ref{int I alpha} relies on  technical computations (Lemma \ref{int 0 Ld}). 

\begin{lem}\label{int 0 Ld}
For $0<T<1$, and  a real number  $p\geq 2$, set $L:=-\frac{\log T}{2\log p}$ and  $K=\min(1,L)$. There exists a  constant $C=C(p,w)>0$, such that
$$J(p, T):=\left(\log T\right)^{-2}\int^{K}_{0}xM\left(T p^{2x}\right)dx\lesssim C\begin{cases}M \left(T\right)\phantom{p^2}\text{ if }\beta\geq 1\text{ or }(\beta< 1, p^{-2}< T<1),\\
M\left(Tp^2\right)\text{ if }\beta< 1, 0<T\leq p^{-2}.
\end{cases} $$ 

\end{lem}

\begin{proof}
When $p^{-2}< T<1$,  the change of variables $u=T p^{2x}$ gives
\begin{align*}
J(p,T)
&=\left(\log T\right)^{-2}\frac{1}{(2\log p)^2} \int^{1}_{T}\log \frac{u}{T}M(u)\frac{du}{u}.
\end{align*}
Since $\log \frac{u}{T}\leq\log \frac{1}{T}$ and $1\leq\frac{1}{u}\leq\frac{1}{T}<p^{2}$,
\begin{align*}
J(p,T)&\leq\left(\log T\right)^{-2}\left(\frac{1}{2\log p}\right)^2 \int^{1}_{T}\log \frac{1}{T}M(u)\frac{1}{u}du\lesssim M(T).
\end{align*}
Next suppose that $0<T\leq p^{-2}$. Since $(\log T)^2\geq 4(\log p)^2$, we notice that
$$J(p,T)\lesssim \int^{1}_{0}xM(T p^{2x})dx \lesssim \begin{cases} \int^{1}_{0}M(T )dx\text{ if }\beta\geq 1 ,\\
\int^{1}_{0}M(T p^{2})dx \text{ if }\beta< 1\end{cases}. $$

\end{proof}

 \begin{proof}[Proof of Lemma \ref{int I alpha}]
Resorting to polar coordinates, and using changes of variables, we have
\begin{align*}
\Jl_{\alpha}
&\leq\int_{Q}\frac{xt^{\alpha}}{\left[x+\sigma\left(p^{x}_{1}\sqrt{t_1},\cdots,p^{x}_{1}\sqrt{t_d}\right)\right]^2} 
\left(\prod^{d}_{k=1} M\left(p^{2x}_{k}t_{k}\right) p^{2x}_{k}\right)dxdt_1\cdots dt_d ,
\end{align*}
where $Q=\left\{(x,t)\in(0,1)\times(0,1)^d,\ \forall k=1..d, 0< t_k<p^{-2x}_{k} \right\}$. 

For  $t=(t_1,\cdots,t_d)\in(0,1)^d$, set 
\begin{align*}
l_k(t)&:=-\frac{\log t_k}{2\log p_k},\ K_k:=\min(1,l_k),\  1\leq k\leq d,\\
l(t)&:=\min_{1\leq k\leq d}l_k(t),\ K:=\min(1,l).
\end{align*}

We observe that 
$Q
=\left\{(x,t)\in(0,1)\times(0,1)^d,\ 0<x<K(t)\right\}$.
Now, for $1\leq k\leq d,$ we set 
$  Q_k:=\left\{(x,t),\ t\in(0,1)^d,\ l(t)=l_k(t),\ 0<x<K_k(t)\right\}.$

Let $(x,t)$ be in $ Q_k$. We have
\begin{equation}\label{tj < Tkj}
	 0<t_l\leq T_{k,l}:=t^{\frac{\log p_l}{\log p_k}}_{k}<1,\ \text{ for }1\leq l\leq d.
\end{equation}
In addition, since $0< x< l_k(t)$, (\ref{tj < Tkj}) implies 
$p^{x}_{l}\sqrt{t_l}<p^{l_k(t)}_{l}\sqrt{t_l}
\leq 1,$
and we see that
$  \frac{1}{\sqrt{t_l} p^{x}_{l}}   \geq p^{l_k(t)-x}_{l}\geq p^{l_k(t)-x}_{1}.$
	Thus
	\begin{align*}
(\log p_d )	\sigma\left(p^{x}_{1}\sqrt{t_1},\cdots,p^{x}_{d}\sqrt{t_d}\right)
&=\log\min_{1\leq l\leq d}\left(\frac{1}{\sqrt{t_l}p^{x}_{l}}\right)\geq\log p_1 \left(l_k(t)-x \right),
\end{align*}	
and
$  x+\sigma\left(p^{x}_{1}\sqrt{t_1},\cdots,p^{x}_{1}\sqrt{t_d}\right)
\gtrsim -\log t_k.$

Set $d\widehat{t_k}=dt_1\cdots dt_{k-1}dt_{k+1}\cdots dt_d,$ and
$$\tilde Q_k:=\left\{(x,t),\ 0<t_k<1,\ 0<t_l<T_{k,l} \text{ for }l\neq k,\  0<x<K_k(t)\right\}.$$
It follows that $\Jl_{\alpha} \lesssim \sum^{d}_{k=1}\Jl_{\alpha,k}$, where
\begin{align*}
\Jl_{\alpha,k}
&=\int_{\tilde Q_k}\frac{x t^{\alpha}}{\left[x+\sigma\left(p^{x}_{1}\sqrt{t_1},\cdots,p^{x}_{1}\sqrt{t_d}\right)\right]^2}\left(\prod^{d}_{l=1} M \left(p^{2x}_{l}t_{l}\right)\right)dx dt.
\end{align*}

We will obtain the Lemma by showing that
\begin{equation}\label{esti J alpha k}
	\Jl_{\alpha,k}\lesssim \prod^{d}_{l=1}\left[{w\left(p^{\alpha_l}_{l}\right)}\right]^{-1}.
\end{equation}

When $\beta\geq 1$, we use that,  for $(x,t)\in\tilde Q_k$, and $l\neq k$, $M\left(p^{2x}_{l}t_{l}\right)\leq M\left(t_{l}\right)$, altogether with Lemma \ref{int 0 Ld}. We derive  (\ref{esti J alpha k}) from
\begin{align*}
\Jl_{\alpha,k}&\lesssim \int_{0<t_k<1}\left(\int_{\prod_{j\neq k}(0,T_{k,j})}t^{\alpha}\int^{K_k(t)}_{0}x\left(\log t_k\right)^{-2}M\left(p^{2x}_{k}t_{k}\right)dx \prod_{l\neq k}M (t_l)d\widehat{t_k} \right)dt_k\\
&\lesssim \int_{0<t_k<1}t^{\alpha_k}_{k}M\left(t_{k}\right)\left(\prod_{j\neq k}\int^{T_{k,j}}_{0}t^{\alpha_j}_{j}M\left(t_{j}\right)dt_j\right)dt_k
\lesssim\prod^{d}_{j=1}\int^{1}_{0}t^{\alpha_j}_{j}M\left(t_{j}\right)dt_j.
\end{align*}

Next, suppose $0<\beta<1$. If  $(x,t)\in \tilde Q_k$, notice that, for $l\neq k$, $t_l p^{2x}_{l}  \leq t_l p^{2 l_k(t)}_{l}\leq 1$; this shows that
$ M\left(p^{2x}_{l}t_{l}\right)\leq M\left(p^{2 l_k(t) }_{l}t_{l}\right). $
Hence, we 
see that $\Jl_{\alpha,k}
\lesssim J_1+J_2,$, where, by  Lemma \ref{int 0 Ld} and the relation $p^{2 l_k(t) }_{l}=T^{-1}_{k,l}$,
\begin{align*}
J_1
&\lesssim\int_{0<t_k < p^{-2}_{k}}t^{\alpha_k}_{k}M(p^{2 }_{k}t_k)\left(\prod_{j\neq k}\int^{T_{k,j}}_{0}t^{\alpha_j}_{j}M\left(t_{j}T_{k,j}^{-1}\right)dt_j \right)dt_k,\\
J_2
&\lesssim\int_{p^{-2}_{k}<t_k <1}t^{\alpha_k}_{k}M(t_k)\left(\prod_{j\neq k}\int^{T_{k,j}}_{0}t^{\alpha_j}_{j}M\left(t_{j}T_{k,j}^{-1}\right)dt_j \right)dt_k.
\end{align*}
A change of variables provides the desired estimate.

\end{proof}

\subsection{Proof of Theorem \ref{boundedness hilbert} (b) and (c)}
If $f(s)= \sum^{+\infty}_{n=1}a_n n^{-s}$ and $g(s)= \sum^{+\infty}_{n=1}b_n n^{-s}$, we have
	$$T_g f(s)=\sum^{\infty}_{n=2}\frac{1}{\log n}\left(\sum_{k|n,k<n}a_k b_{n/k}\right)n^{-s}.$$

As in the case of $\Hl^2$, the operator 
$$ a_1+ \sum^{\infty}_{n=2}a_n n^{-s}\mapsto a_1+\sum^{\infty}_{n=2}a_n (\log n)^{-1}n^{-s}$$
is compact on $\Hl_w$. Thus,  set $b_1=1$, and  our study will be unchanged if we replace $T_g$  by 
$$ \tilde{T}_g f(s)=\sum^{\infty}_{n=2}\frac{1}{\log n}\left(\sum_{k|n}a_k b_{n/k}\right)n^{-s}.$$

\begin{lem}\label{bd H2 implies bd Hw}
If
 $T_g $ is bounded on $\Hl^2$, then $g$ is in $\Xw$, and the operator norms satisfy
$$  \left\|T_g\right\|_{\Ll(\Hl^2_w)}\leq\left\|T_g\right\|_{\Ll(\Hl^2)}.$$
\end{lem}

\begin{proof}
If  $f(s)= \sum^{+\infty}_{n=1}a_n n^{-s}$ is in $\Hl^2_w$,  the function $\tilde f(s)= \sum^{+\infty}_{n=1}a_n w^{-1/2}_{n} n^{-s}$ is in $\Hl^2$ and
 $\left\|f\right\|_{\Hw}=\left\|\tilde f\right\|_{\Hl^2}$.
 Since $w_k\leq w_{kl}$ for any  integers $k,l$, the Lemma is proven by
 the inequality
\begin{align*}
\left\|T_g f\right\|^{2}_{\Hl^2_w}\leq\sum^{\infty}_{n=2}\left(\log n\right)^{-2}\left|\sum_{k|n,k<n}w^{-1/2}_{k}a_k b_{n/k}\right|^2=\left\|T_g \tilde f\right\|^{2}_{\Hl^2}.
\end{align*}

\end{proof}

We will also use the sufficient condition proved in Theorem 2.3 in \cite{bre-perf-seip}, stating that if $g$ is in $BMOA(\C_0)\cap\Dl$, then $T_g$ is bounded on $\Hl^2$, with 
\begin{equation}\label{norm Tg H2 bmo}
	\left\|T_g\right\|_{\Hl^2}\lesssim \left\|g\right\|_{BMOA(\C_0)}.
\end{equation}

\begin{proof}[Proof of Theorem \ref{boundedness hilbert} (b) and (c)]
If $g$ is in $BMOA(\C_0)$, $T_g$ is bounded on $\Hl^2$, and (b) is a consequence of (\ref{norm Tg H2 bmo}) and Lemma \ref{bd H2 implies bd Hw}.

To prove (c), we use that  $(T_g f)'  =fg'$, and that $\Hw$ is embedded in $A_{i,\delta}\left(\C_{1/2}\right)$, with $\delta =\delta(w)>0.$   We set
$$  d\nu_g(s)=\left|g'(s)\right|^2 \frac{\left(\sigma-\frac{1}{2}\right)^{\delta+1}}{\left|s+\frac{1}{2}\right|^{2\left(\delta+1\right)}}dV(s).$$  Now formula  (\ref{stanton}),    the boundedness of $T_g$ on  $\Hw$ and  Lemma  \ref{embedding} induce that
$$ \int_{\C_{1/2}}\left|f(s)\right|^2 d\nu_g(s)\lesssim \left\|T_g f\right\|^{2}_{A_{i,\delta}\left(\C_{1/2}\right)} \leq c\left(w\right)\left\|T_gf\right\|^{2}_{\Hw}\leq c\left(w\right)\left\|T_g\right\|^{2}_{\Ll\left(\Hw\right)}\left\|f\right\|^{2}_{\Hw} ,$$
Thus, $\nu_g$ is a Carleson measure for $\Hw$ and
 $\left\|\nu_g\right\|_{CM\left(\Hw\right)}\lesssim \left\|T_g\right\|^{2}_{\Ll\left(\Hw\right)}$. By Lemma \ref{CM B alpha A beta}, $\nu_g$ is also a Carleson measure for $A_{i,\delta}\left(\C_{1/2}\right)$ 
 and 
$$\left\|\nu_g\right\|_{CM\left(A_{i,\delta }\left(\C_{1/2}\right)  \right)}
\lesssim \left\|T_g\right\|^{2}_{\Ll\left(\Hw\right)}.
$$ 
 We conclude by the characterization of the Bloch space given in Lemma \ref{charact Bloch CM}. 


\end{proof}

We  get a result which is  in agreement with the situation for Hardy spaces \cite{Cima}, Bergman spaces \cite{ale-sisk}
or the Hardy  space of Dirichlet series $\Hl^2$ \cite{bre-perf-seip}, with the same  proof.

\begin{cor}\label{Tg bd implies g Lp}
 If $g$ is in  $\Xw$, then $g$ is in $\cap_{0<p<\infty}\Hl^p_{w}$, and 
 there exists $c>0$, such that the function $e^{c\left|\Bl g\right|}$ is integrable on $\D^{\infty}$, with respect to $d \mu_{w}$.
\end{cor}

\section{Compactness }\label{compa}
We now present a little oh version of Theorem  \ref{boundedness hilbert}.

  If the symbol is a vector of the standard orthonormal basis of $\Hw
$, that is
$$ g(s)= e_{w,n}(s):= w^{1/2}_{n}n^{-s},
$$
the operator $T^{*}_{g}T_g$ is diagonal, and its eigenvalues
$$  \lambda_{n,k}^{2}=\frac{w_n w_k}{w_{nk}}\left(\frac{\log n}{\log n+\log k}\right)^2$$
tend to $0$ as $k\rightarrow +\infty$. Thus $T_g$ is compact. It follows that every Dirichlet polynomial generates a compact Volterra operator on $\Hw$.

\subsection{Case when $\Bl g$ depends on a finite number of variables}

We  approximate a symbol $g$ which is in $\text{Bloch}_0({\Bbb C}_{0})\cap \Dl_{d}$ by a Dirichlet polynomial $P$ in the $\text{Bloch}({\Bbb C}_{0})$-norm. From Theorem \ref{boundedness hilbert} (a), $T_g$ is approximated in the operator norm by the compact operator $T_P$.

\begin{thm}\label{compact d variables}
If $g$ is in $\text{Bloch}_0({\Bbb C}_{0})\cap \Dl_{d}$, then $T_g$ is compact on $\Hw$.
\end{thm}
\subsection{Sufficient / necessary conditions for compactness}
In general, if the symbol $g(s)=\sum_{n\geq 2}b_n n^{-s}$ satisfies an inequality of the form $\left\|T_g\right\|^{2}_{\Ll(\Hw)}\leq\sum_{n\geq 2} \left|b_n\right|^2 W(n)<\infty$, we approximate $T_g$ in the operator norm by the compact operator $T_{S_N g}$. Therefore, $T_g$ is compact (see \cite{bre-perf-seip}).

The little oh version of Theorem \ref{boundedness hilbert} is related to the properties of $VMOA(\C_0)\cap \Dl$, and with the concept of vanishing Carleson measures.

\begin{thm}\label{compactness}\label{cpt}
Let  $g$ be in $\Dl$.

\begin{enumerate}
	\item If $g$ is in $VMOA(\C_0)\cap \Dl$, then $T_g$ is compact on $\Hw$.
	
	\item   If  $T_g$ is compact on $\Hw$, then $g$ is in $\text{Bloch}_0(\C_{1/2})$.
	
	\end{enumerate}
\end{thm}

\begin{proof}
In order to prove (1), we use that $VMOA(\C_0)\cap \Dl$ is the closure of Dirichlet polynomials in $BMOA(\C_0)$ (see \cite{bre-perf-seip}), and that, from Theorem  \ref{boundedness hilbert}, we have $\left\|T_g\right\|_{\Ll(\Hw)}\lesssim \left\|g\right\|_{BMOA(\C_0)}$.

Recall that $\Hw$ is embedded in $A_{i,\delta}(\C_{1/2})$,  $\delta=\delta(w)$ being defined in (\ref{def delta}). Assume that $T_g$ is compact on $\Hw$, and consider the measure 
$$ d\nu_g(s)=\left|g'(s)\right|^2 \frac{\left(\sigma-\frac{1}{2}\right)^{\delta+1}}{\left|s+\frac{1}{2}\right|^{2(\delta+1)}}dV(s).$$
 Let $(f_k)_k$ be a weakly compact sequence in $\Hw$. Formula  (\ref{stanton}),  and  Lemma  \ref{embedding} imply  that
$$  \int_{\C_{1/2}}\left|f_k(s)\right|^2  d\nu_g(s)\asymp \left\|T_g f_k\right\|^{2}_{A_{i, \delta}(\C_{1/2})}\lesssim \left\|T_g f_k\right\|^{2}_{\Hw}.$$
By the  compactness of $T_g$, $\nu_g $ is a vanishing Carleson measure for $A_{i,\delta}(\C_{1/2})$, with
$$\lim_{k\rightarrow \infty} \int_{\C_{1/2}}\left|f_k(s)\right|^2  d\nu_g(s)=0. $$
 Now, $g$ is in $\text{Bloch}_0(\C_{1/2})$, by the characterization of vanishing Carleson measures  (Lemma \ref{CM bergman}).

\end{proof}

\section{Membership in Schatten classes}\label{scha}
Let $g$ be a non constant symbol. 
 As in the case of $\Hl^2$, the Volterra operator $T_g$ on $\Hw$ does not belong to any Schatten class.

\begin{thm}\label{Schatten}
If the Dirichlet series $g(s)=\sum_{n\geq 2}b_n n^{-s}$ is not $0$, then $T_g:\Hw\rightarrow\Hw$ is not in the Schatten class $\Sl_p$, for any $0<p<\infty$.
\end{thm}

\begin{proof}
Recall that $(e_{w,n})_n$ is an orthonormal basis of $\Hw$. We follow the  reasoning of Theorem 7.2 \cite{bre-perf-seip}. Using that $w_{Nn}\lesssim w_{N}w_{n}$, we see that, for $N\geq n$, 
\begin{align*}
\left\|T_g e_{w,n}\right\|^{2}_{\Hw}&=\sum^{+\infty}_{k=2}\frac{\left|b_k\right|^2 (\log k)^2}{(\log (kn))^2}\frac{w_n}{w_{kn}}
\geq \frac{\left|b_N\right|^2 (\log N)^2}{(\log (Nn))^2}\frac{w_n}{w_{Nn}}\gtrsim  \frac{\left|b_N\right|^2 (\log N)^2}{(2\log n)^2}\frac{1}{w_{N}}.
\end{align*}
For $p\geq 2$, we obtain 
$$  \left\|T_g \right\|^{p}_{\Sl_p}\geq\sum^{+\infty}_{n=N}\left\|T_g e_{w,n}\right\|^{p}_{\Hw}=+\infty.$$
Therefore $T_g$ is not in $\Sl_p$ for $p\geq 2$, neither for $0<p<\infty$.
\end{proof}

\section{Examples}\label{ex}
 In this section, we study the boundedness of $T_g$ on $\Hw$
, for specific symbols $g$. 
We consider fractional primitives of translates of the weighted Zeta function $Z_w$ and homogeneous  symbols, which are the counterparts of the symbols  presented in \cite{bre-perf-seip} in the $\Hl^2$ setting. 
The techniques of proof, as well as the results are  similar to theirs, and we omit the details.

\subsection{Fractional primitives of translates of $Z_w$}
\begin{prop}\label{primitive Zw}
With the notation of (\ref{def delta}), take  $1/2\leq a<1$, $2\gamma>\delta(w)-1$. If
$$ g(s)=\sum^{\infty}_{n=2}w_n\frac{n^{-a}}{\left(\log n\right)^{\gamma+1}} n^{-s}, $$
 then $T_g$ is unbounded on  
$\Hw$.
\end{prop}

\begin{proof}
 Abel summation and the Chebyshev estimate induce that $g$ is in $\Hw$. 
If $f (s)=\sum^{\infty}_{n=1}a_n n^{-s}$, and $g(s)=\sum^{\infty}_{n=2}\frac{b_n}{\log n}n^{-s}$, we set $A_n=\sum_{k|n}a_{n/k}b_k$, so that
$$ \left\|\tilde{T}_{g}f\right\|^{2}_{\Hw}=\sum^{\infty}_{n=2}\frac{1}{(w_n\log n)^2}A^2_n. $$
 We adapt  the test functions of \cite{bre-perf-seip}, and  take $f_J(s)=\prod^{J}_{j=1}\left(1+w^{1/2}_{2}p^{-s}_{j}\right)$, for $J\geq 1$. By construction, it satisfies $\left\|f_J\right\|_{\Hw}\asymp 2^{J/2}$. Now, for $\mathcal{J}$  a non-empty subset of $\left\{1,\cdots, J\right\}$,  we set $n_{\mathcal{J}}=\prod_{j\in\mathcal{J}}p_j$, and  observe that
\begin{align*}
A_{n_{\mathcal{J}}} =\sum_{1\leq k\leq\left|\mathcal{J}\right|,\left\{ p_{j_1},\cdots, p_{j_k}\right\}\subset\mathcal{J} }w^{\frac{\left|\mathcal{J}\right|-k}{2}}_{2}\left[\log\left(p_{j_1 }\cdots p_{j_k }\right)\right]^{-\gamma}w^{k}_{2}\left(p_{j_1 }\cdots p_{j_k }\right)^{-a}+w^{\frac{\left|\mathcal{J}\right|}{2}}_{2}.
\end{align*}
First assume that $\gamma\geq 0$. From the prime number Theorem, we obtain that
\begin{align*}
A_{n_{\mathcal{J}}}
&\gtrsim w^{\frac{\left|\mathcal{J}\right|}{2}}_{2}\left[J\log J\right]^{-\gamma}
\left[1+\sum_{1\leq k\leq\left|\mathcal{J}\right|  ,\left\{ p_{j_1},\cdots, p_{j_k}\right\}\subset\mathcal{J}  }
w_2^{k/2}\left(p_{j_1 }\cdots p_{j_k }\right)^{-a}\right].
\end{align*}
Therefore, it follows again  from the  prime number Theorem that
\begin{align*}
\left\|\tilde{T}_{g}f_J\right\|^{2}_{\Hw}
&\gtrsim\sum_{\mathcal{J}\subset \left\{1,\cdots, J\right\},\left|\mathcal{J}\right|\geq J/2 }\frac{1}{\left(\log n_{\mathcal{J}}\right)^2}\left[J\log J\right]^{-2\gamma}
\prod_{j\in\mathcal{J}}\left(1+w_2^{1/2} p_j^{-a}\right)^2\\
&\gtrsim 2^{J-1} \left[J\log J\right]^{-2\gamma} \min_{\left|\mathcal{J}\right|\geq J/2 }\frac{1}{\left(\log n_{\mathcal{J}}\right)^2}\prod_{j\in\mathcal{J}}\left(1+ w_2^{1/2}  p_j^{-a}\right)^2\\
&\gtrsim e^{cJ^{1-a}\left(\log J\right)^{-a}}\left\|f_J\right\|^{2}_{\Hw},
\end{align*}
for some  constant $c>0$, and $T_g$ is unbounded. The case when $\gamma<0$ is similar.
\end{proof}

\subsection{Homogeneous symbols}
An m-homogeneous Dirichlet series has the form
$$	g(s)=\sum_{\Omega(n)=m}b_n n^{-s}.$$

We extend Theorem 4.2 in \cite{bre-perf-seip}  to the spaces $\Hw$.

\begin{prop}\label{homo bd}
There exist weights $W_m(n)$ such that for $g(s)=\sum_{\Omega(n)=m} b_n n^{-s},$
\begin{equation}\label{norm Tg homo}
	\left\|T_g\right\|_{\Ll(\Hw)}\leq\left(\sum_{\Omega(n)=m}\left|b_n\right|^2 W_m(n)\right)^{1/2}.
\end{equation}
Precisely, there exist  absolute constants $C_m$ for which
$$ W_m(n)=\begin{cases}
C_1 \phantom{111111111}\text{ for }m=1,\\
C_2 \frac{\log n}{\log_2 n} \phantom{1111}\text{ for }m=2,\\
C_m \frac{ n^{\frac{m-2}{m}}}{\left(\log n\right)^{m-2}}\text{ for }m\geq 3.
\end{cases} $$

Moreover, when  $m=2$, $\log_2 n$ cannot be replaced in (\ref{norm Tg homo}) by $\left(\log_2 n\right)^{1+\ep}$ for any $\ep>0.$
\end{prop}

\begin{proof}
If a linear symbol ($m=1$)
$	g(s)=\sum_{p\in\mathbb{P}}b_p p^{-s}$
belongs to $\Hl^2$, we observe that $\left\|  g\right\|^2_{\Hl^2}=2^{\b}  \left\| g \right\|^{2}_{\Bl^{2}_{\b}}=\left(\beta+1\right)\left\|g\right\|^{2}_{\Al^{2}_{\beta}}$ . Hence, it follows from  Theorem 4.1 in \cite{bre-perf-seip} and Lemma \ref{bd H2 implies bd Hw} that 
$T_{g}$ is bounded on $\Hw$  and
$  \left\|T_{g}\right\|_{\Ll(\Hw)}\leq\left\|T_g\right\|_{\Ll(\Hl^2)}.$
 One can choose $C_1=\max\left(\left(\b+1\right)^{-1}, 2^{-\b}\right)$.

(\ref{norm Tg homo}) is a consequence of Theorem 4.2 in \cite{bre-perf-seip} and Lemma \ref{bd H2 implies bd Hw}. We now  prove the sharpness of the factor $\log_2 n$. We assume that for some $\ep>0$, every $2$-homogeneous Dirichlet series $g$ satisfies
\begin{equation}\label{sharpness 2}
	\left\|T_{g}\right\|_{\Ll(\Hw)}\leq C_2 \left(\sum_{\Omega(n)=m}\left|b_n\right|^2  \frac{\log n}{\left(\log_2 n\right)^{1+\ep}}\right)^{1/2}.
\end{equation}
For $x$ a large real number, and $q\sim e^x$ a prime number, the symbol considered in \cite{bre-perf-seip} is 
$$  g_x(s)=\sum_{x/2<p\leq x}\frac{\left(\log_2 (pq)\right)^{1+\ep/2}}{p}\left(pq\right)^{-s}.$$
We take as test functions 
$$  f_x(s)=\sum^{+\infty}_{n=1}a_n n^{-s}=\prod_{x/2<p\leq x}\left(1+ w^{1/2}_{2} p^{-s}\right).$$
If $S_x$ denotes the set of square-free integers generated by the primes $x/2<p\leq x$, we have $\left\|f_x\right\|^{2}_{\Hw}\asymp\left|S_x\right|=2^{N(x)}$, where $N(x):=\pi(x)-\pi(x/2)$. Now,
$$ \frac{\left\|T_{g_x}f_x\right\|^2_{\Hw}}{\left\|f_x\right\|^2_{\Hw}}\gtrsim \frac{1}{\left|S_x\right|} \sum_{n\in S_x}w^{-1}_{nq}\left(\log (nq)\right)^{-2}\left|\sum_{pq|nq}\log (pq)\frac{\left(\log_2 (pq)\right)^{1+\ep/2}}{p}a_{n/p}\right|^2.$$
If $n\in S_x$, and $p|n$, we have
 $a_{n/p}=w_2^{\frac{1}{2}\left[\omega(n)-1\right]},\ w_n=w^{\omega(n)}_{2},$
 and $w_{nq}=w_{n} w_{q}$.
Thus,  
$$ \frac{\left\|T_{g_x}f_x\right\|^2_{\Hw}}{\left\|f_x\right\|^2_{\Hw}} \gtrsim \frac{1}{\left|S_x\right|}\frac{\left(\log x\right)^{2+\ep}}{x^2} \sum_{n\in S_x}\omega(n)^2. $$
Now $\sum_{n\in S_x}\omega(n)^2=\sum^{N(x)}_{k=1}\binom{N(x)}{k}k^2\asymp N(x)^2 2^{N(x)}$, and  (\ref{sharpness 2}) does not hold, due to
$$ \frac{\left\|T_{g_x}f_x\right\|_{\Hw}}{\left\|f_x\right\|_{\Hw}} \gtrsim\left(\log x\right)^{\ep}.$$

\end{proof}

We will exhibit an homogeneous symbol $g$ which is in $\Hw\cap\text{Bloch}_0(\C_{1/2})$, but not in $\Xw$. In fact, we  observe that $g$ is in every $\Hpa$.

\begin{lem}\label{homogeneous in each Hp}
If $g$ is an $m$-homogeneous Dirichlet series  in $\Hw$, then $g$ is in $\cap_{0<p<\infty}\Hl^{p}_{w}$ and, for any $0<p<\infty$, there exists $c=c(m,p)$ such that
\begin{equation}\label{homo Hp H2}
	\left\|g\right\|_{\Hl_w^{p}} \leq c\left\|g\right\|_{\Hw}.
\end{equation}
 
\end{lem}

\begin{proof}
It is enough to consider the case $p\geq 2$. We first prove the inequality for $p=2^k$, $k$ being a positive integer, by an induction argument. 

Obviously, it holds for $k=1$.

 Our proof is inspired of Lemma 8 in \cite{seip}.
 For any integer $m$, there exists a constant  $C(m)$, such that $\max\left( w_n, d(n)\right)\leq C(m)$, whenever  $ \Omega(n)=m$.

 If $f(s)=\sum_{n}a_n n^{-s}$ is $m$-homogeneous, then $f^2(s)=\sum_{n}b_n n^{-s}$ is $2m$-homogeneous, and 
$\left|b_n\right|^2\leq d(n) \sum_{k|n}\left|a_k\right|^2\left|a_{n/k}\right|^2 .$
Since $w_n\geq\sqrt{w_k}\sqrt{w_{n/k}},$
\begin{align*}
\left\|f\right\|^{4}_{\Hl^{4}_{w}}&=\left\|f^2\right\|^{2}_{\Hl^{2}_{w}}\leq\sum_{\Omega(n)=2m}d(n)w^{-1}_{n} \left(\sum_{k|n}\left|a_k\right|^2\left|a_{n/k}\right|^2\right)
\leq C(2m)\sum_{\Omega(n)=2m}\left(\sum_{k|n}\frac{\left|a_k\right|^2}{\sqrt{w_k}}\frac{\left|a_{n/k}\right|^2}{\sqrt{w_{n/k}}} \right)\\
&=C(2m)\left(\sum_{k}\frac{\left|a_k\right|^2}{\sqrt{w_k}}\right)^2
\leq C(2m) C(m)\left\|f\right\|^{4}_{\Hl^{2}_{w}}.
\end{align*}
Now, suppose that, for some $k$,   an $m$-homogeneous Dirichlet series $h$ satisfies 
$$  \left\|h\right\|^{2^k}_{\Hl^{2^k}_{w}}\leq K(m,k) \left\|h\right\|^{2^k}_{\Hl^{2}_{w}}\text{ for  any }m.$$
We obtain that
\begin{align*}
 \left\|f\right\|^{2^{k+1}}_{\Hl^{2^{k+1}}_{w}}&=  \left\|f^2\right\|^{2^k}_{\Hl^{2^k}_{w}}\leq K(2m,k) \left\|f^2\right\|^{2^k}_{\Hl^{2}_{w}}= K(2m,k) \left\|f\right\|^{2^{k+1}}_{\Hl^{4}_{w}}\\
&\leq  K(2m,k)\left[C(2m) C(m)\left\|f\right\|^{4}_{\Hl^{2}_{w}}\right]^{2^{k-1}}.
\end{align*}
For general $p$, (\ref{homo Hp H2}) is a consequence of H\"older's inequality.
\end{proof}

For our construction, we need two technical Lemmas.

\begin{lem}\label{hj}
Assume that $0<\delta<1$ and $0<\e$.
For $j=1..3$, we set $h_j(s)=\sum_{p\geq 3}\alpha_{j,p} p^{-s}$,  where  
$$ \alpha_{1,p}=\left(\log_2 p\right)^{-\delta},\ \alpha_{2,p}=\log_2 p,\ \alpha_{3,p}={\log p}{\left(\log_2 p\right)^{-\eta}}. $$
For a real number  $\sigma>1$, set $ \s':=\frac{1}{\s-1}$. Then we have
\begin{align}\label{esti hj j=123}
	h_1(\s)\asymp \left(\log \s'\right)^{1-\d};\
	h_2(\s)\asymp\log_2\left(\s'\right);\
	h_3(\s)\asymp\s'\left(\log \s'\right)^{-\e},\ \text{ as }\s\rightarrow 1^+.
\end{align}
\end{lem}

\begin{proof}
These asymptotics will follow from computations inspired by \cite{ivic,ap}. Recall that
\begin{equation}\label{esti sum 1/p}
A_1(t):=\sum_{3\leq p\leq t}\frac{1}{p}= \log_2 t+O(1).
  \end{equation}

Setting  $f_1(t)=\frac{t^{-(\s-1)}}{\left(\log_2 t\right)^{\delta}}$, we have
\begin{align*}
 h_1(\s)&= \sum_{p\geq 3}\frac{p^{-(\s-1)}}{p\left(\log_2 p\right)^{\delta}}=-\int^{+\infty}_{ 3}A_1(t) f'_1(t)dt+O(1)\\
&\asymp(\s-1)\int^{+\infty}_{3}\left(\log_2 t\right)^{1-\delta}t^{-\s}dt
=(\s-1)\left(\int^{\ss}_{\log 3}+\int^{+\infty}_{\ss}\right)\left(\log x\right)^{1-\delta}e^{-(\s-1)x}dx.
\end{align*}
Using integration by parts (for the first integral), and a change of variable (for the second one), we obtain
\begin{align*}
h_1(\s)&\asymp(\s-1) \int^{\ss}_{\log 3}\left(\log x\right)^{1-\delta}dx+\int^{+\infty}_{1}\left(\log y+\log\s'\right)^{1-\delta}e^{-y} dy\\
&\asymp(\s-1)\left[x\left(\log x\right)^{1-\delta}\right]^{x=\ss}_{x=\log 3}+\int^{+\infty}_{1}\left[\left(\log y\right)^{1-\delta}+\left(\log\ss\right)^{1-\delta}\right]e^{-y} dy \\
&\asymp\left(\log\s'\right)^{1-\delta}.
\end{align*}

The functions $h_2$ and $h_3$ are handled similarly.
For $x\geq 3$, summation by parts  and (\ref{esti sum 1/p}) induce that, 
$$  A_2(x):=\sum_{3\leq p\leq x}\frac{1}{p\log_2 p}=\frac{A_1(x)}{\log_2 x}+\int^{x}_{3}\frac{A_1(t)}{t \log t (\log_2 t)^2}dt+O(1)\asymp \log_3 x.$$
Set $f_2(t):=t^{-(\s-1)} $. Then,
\begin{align*}
h_2(\s)&\asymp-\int^{+\infty}_{3}A_2(t)f'_2(t)dt +O(1)\asymp (\s-1)\int^{+\infty}_{3}(\log_3 t) t^{-\s}dt\\
&=(\s-1)\left(\int^{e\ss}_{\log 3}+\int^{+\infty}_{e\ss}\right)(\log_2 x )e^{-(\s-1)x}dx. 
\end{align*}
Now 
$$  (\s-1)\int^{e\ss}_{\log 3}(\log_2 x) e^{-(\s-1)x}dx\asymp (\s-1)\int^{e\ss}_{\log 3}(\log_2 x) dx\leq(\s-1)e\ss\left(\log_2 \left(e\ss\right)\right)\lesssim \log_2\ss.$$
We perform a change of variable in the integral over $[e \ss,+\infty)$.
\begin{align*}
I_{2,2}&:=(\s-1)\int^{+\infty}_{e\ss}(\log_2 x) e^{-(\s-1)x}dx=\int^{+\infty}_{e}\left[\log \left(\log y+\log \ss\right)\right]e^{-y}dy\\
&\geq (\log_2 \ss)\int^{+\infty}_{e}e^{-y}dy\gtrsim\log_2 \ss.
\end{align*}
Since $ \lo(a+b)\leq\lo a \lo b+1, \ \text{for }a\geq e \text{ and }b\geq e $,  
we obtain
\begin{align*}
I_{2,2}&\leq\int^{+\infty}_{e}\left[(\lo_2 y)( \lo_2\ss)+1\right]e^{-y}dy\lesssim \lo_2\ss,
\end{align*}
and $I_{2,2}\asymp \lo_2\ss$. It follows that $h_2(\s)\asymp \lo_2\ss.$

We now focus on $h_3$.
By Mertens'  first Theorem, 
$ A_3(x):= \sum_{3\leq p\leq x}\frac{\lo p}{p}=\log x +O(1),$
and putting $  f_3(t):=t^{-(\s-1)}\left(\lo_2 t\right)^{-\e},$
we see that
\begin{align*}
h_3(\s)& =-\int^{+\infty}_{3}A_3(t)f'_3(t)dt+O(1)\asymp(\s-1)\int^{+\infty}_{3}\left(\lo t \right)t^{-\s}\left(\lo_2 t\right)^{-\e}dt\\
&\asymp(\s-1)\left(\int^{\ss}_{\lo 3}+\int^{+\infty}_{\ss}\right)x e^{- (\s-1)x} \left(\lo x\right)^{-\e}dx.
\end{align*}
Integration by parts gives that
\begin{align*}
I_{3,1}&:=(\s-1)\int^{\ss}_{\lo 3}x e^{- (\s-1)x} \left(\lo x\right)^{-\e}dx\asymp(\s-1)\int^{\ss}_{\lo 3}x  \left(\lo x\right)^{-\e}dx
\asymp\ss \left(\lo \ss\right)^{-\e}.
\end{align*}
Next,  (\ref{esti hj j=123}) is a consequence of
\begin{align*}
I_{3,2}&:=
(\s-1)\int^{+\infty}_{\ss}x e^{- (\s-1)x} \left(\lo x\right)^{-\e}dx=\frac{1}{\s-1}\int^{+\infty}_{1}y e^{-y}\left(\lo y+\lo \ss\right)^{-\e}dy\\
&\lesssim\ss \int^{+\infty}_{1}\frac{y e^{-y}}{\left(\lo \ss\right)^{\e}}dy.
\end{align*}
\end{proof}

\begin{lem}\label{S on p1p2p3}
If $2\e>1$ and $\d+\e>1$, we have
$$  S:=\sum_{p_1,p_2,p_3\in\PP, p_j\geq 3}\frac{1}{p_1p_2p_3\left(\lo_2 p_1\right)^{2\d}\left(\lo_2 p_2\right)^{2}}\frac{\left(\lo p_3\right)^2}{ \left(\lo_2 p_3\right)^{2\e}\left(\lo(p_1 p_2 p_3)\right)^2}<\infty.$$
\end{lem}

\begin{proof}
For $ p_1,p_2\geq 3,$ we set
$ L:=\lo(p_1p_2)$ and $S_3(p_1,p_2):=\sum_{p_3}\frac{\left(\lo p_3\right)^2}{p_3 \left(\lo_2 p_3\right)^{2\e}\left(\lo p_3+L\right)^2}.$ Then, we have
$$S=\sum_{p_1,p_2,p_3}\frac{1}{p_1p_2\left(\lo_2 p_1\right)^{2\d}\left(\lo_2 p_2\right)^{2}}S_3(p_1,p_2).$$
Under the condition $2\e>1$, the sum $S_3(p_1,p_2)$
 converges, and  
\begin{align*}
S_3(p_1,p_2)&=-\int^{+\infty}_{3}A_1(t)\frac{d}{dt}\left[ \frac{(\log t)^2}{\left(\lo_2 t\right)^{2\e}(\lo t+L)^2} \right]dt+\frac{O(1)}{L^2}\\
&\lesssim \frac{O(1)}{L^2}+\int^{+\infty}_{3}\frac{\lo t}{t(\lo_2 t)^{2\e}(\lo t+L)^2}dt=\frac{O(1)}{L^2}+\left(\int^L_{\lo 3}+\int^{+\infty}_L\right)\frac{xdx}{\left(\lo x\right)^{2\e}\left(x+L\right)^2}.
\end{align*}
Integration by parts gives
\begin{align*}
I_{3,1}&:=\int^L_{\lo 3}\frac{xdx}{\left(\lo x\right)^{2\e}\left(x+L\right)^2}\asymp \frac{1}{L^2}\int^L_{\lo 3}\frac{xdx}{\left(\lo x\right)^{2\e}}
\asymp \left(\lo L\right)^{-2\e}.
\end{align*}
We handle the second integral via a change of variable:
\begin{align*}
I_{3,2}&:=\int^{+\infty}_{L}\frac{xdx}{\left(\lo x\right)^{2\e}\left(x+L\right)^2}
=\left(\int^{L}_{1}+\int^{+\infty}_L\right)\frac{ydy}{\left(1+y\right)^2\left(\log y+\lo L\right)^{2\e}}\\
&\lesssim \frac{1}{(\lo L)^{2\e}}\int^{L}_{1}\frac{dy}{y}+\int^{+\infty}_L\frac{dy}{y(\lo y)^{2\e}}\asymp \left(\lo L\right)^{1-2\e}.
\end{align*}
Therefore 
$$S_3(p_1,p_2)\lesssim\left(\lo L\right)^{1-2\e},\ L=\lo(p_1p_2).$$
We next put $M=\lo p_1$, and deal with 
$$  S_2(p_1):=\sum_{p_2}\frac{1}{p_2(\lo_2 p_2)^2}S_3(p_1,p_2)\lesssim \sum_{p}\frac{1}{p(\lo_2 p)^2\left[\lo\left(\lo p+M\right)\right]^{2\e-1}}.$$
With the notation
$  f_2(t):=\left[{\left(\lo_2 t\right)^2\left[\lo\left(\lo t+M\right)\right]^{2\e-1}}\right]^{-1},$ we 
obtain that
\begin{align*}
S_2(p_1)&=\frac{O(1)}{\left(\lo M\right)^{2\e-1}}-\int^{+\infty}_{3}A_1(t)f'_2(t)dt
\lesssim\frac{O(1)}{\left(\lo M\right)^{2\e-1}}+I_{2,1}+ I_{2,2}, 
\end{align*}
 where 
\begin{align*}
I_{2,1}&:= \int^{+\infty}_{3}\frac{dt}{t\lo t\left(\lo_2 t\right)^2\left[\lo\left(\lo t+M\right)\right]^{2\e-1}} ;\\
 I_{2,2}&:= \int^{+\infty}_{3}\frac{dt}{t\left(\lo_2 t\right)\left(\lo t+M\right)\left[\lo\left(\lo t+M\right)\right]^{2\e}}.
\end{align*}
We derive
\begin{align*}
I_{2,1}&=\left(\int^{M}_{\lo 3} +\int^{+\infty}_{M}\right)\frac{dx}{x\left(\lo x\right)^2\left[\lo\left(x+M\right)\right]^{2\e-1}}\\
&\le \frac{1}{\left[\lo M\right]^{2\e-1}}\int^{M}_{\lo 3}\frac{dx}{x\left(\lo x\right)^2}+\left(\lo M\right)^{1-2\e}\int^{+\infty}_{M}\frac{dx}{x\left(\lo x\right)^2}
\le \left(\lo M\right)^{1-2\e}.
\end{align*}
The second integral is estimated in the same way:
\begin{align*}
I_{2,2}&=\left(\int^{M}_{\lo 3} +\int^{+\infty}_{M}\right)\frac{dx}{(x+M)(\lo x)\left[\lo(x+M)\right]^{2\e}}\\
&\le \frac{1}{M(\lo M)^{2\e}}\int^{M}_{\lo 3}\frac{dx}{\lo x}+\frac{1}{(\lo M)^{2\e-1}}\int^{+\infty}_{M}\frac{dx}{x(\lo x)^2}\\
&\asymp \frac{1}{M(\lo M)^{2\e}}\left(\left[\frac{x}{\lo x}\right]^{x=M}_{x=\lo 3}+\int^{M}_{\lo 3}\frac{x^2}{2}\frac{(\lo x)^{-2}}{x}dx\right)+	\frac{1}{(\lo M)^{2\e}}\asymp \frac{1}{(\lo M)^{2\e}}.
\end{align*}
Therefore, we have 
$$  S_2(p_1)\le \frac{1}{(\lo M)^{2\e-1}},\ M=\lo p_1.$$
It follows that
$$  S\le \sum_{p_1}\frac{1}{p_1(\lo_2 p_1)^{2\d}}S_2(p_1)\le \sum_{p\geq 3}\frac{1}{p(\lo_2 p)^{\ep}},\ \ep:=2\d+2\e-1.$$
Again, partial summation gives that when $\ep>1$, 
$$  \sum_{3\leq p}\frac{1}{p(\lo_2 p)^{\ep}}\asymp\ep\int^{+\infty}_{3}\frac{\lo_2 t}{t(\lo t)(\lo_2 t)^{\ep+1}}dt<\infty.$$
\end{proof}
\begin{prop}\label{ctex m=3}
There exists a $3$-homogeneous function $g$ which is in $\left(\cap_{0<p<\infty}\Hpa\right)\cap\text{Bloch}_0(\C_{1/2})$, such that $T_g$ is unbounded on $\Hw$. 

\end{prop}

\begin{proof}
Using Lemma \ref{hj}, we see that, if $g'=-(h_1h_2h_3)_{\frac{1}{2}}$, $g'$ is convergent on $\C_{1/2}$, and its estimate near the line $\Re s=\frac{1}{2}$ is determined by the behavior of the functions $h_j$ near the line $\Re s=1$. Then 
 $g$ is in $\text{Bloch}_0(\C_{1/2})$, because of
$$ \left| g'(\s)\right|\asymp \frac{1}{\s-\frac{1}{2}}\left(\lo \frac{1}{\s-\frac{1}{2}}\right)^{1-\d-\e}\left(\lo_2\frac{1}{\s-\frac{1}{2}}\right),\ \text{ as }\s\rightarrow 1^+.$$
On another hand,
 the $3$-homogeneous function 
$$  g(s)=\sum_{n}b_n n^{-s}=\sum_{p_1,p_2,p_3}\frac{\alpha_{1,p_1}\alpha_{2,p_2}\alpha_{3,p_3}}{\log(p_1p_2p_3)}\left(p_1p_2p_3\right)^{-s}$$
is in $\Hw$ by Lemma \ref{S on p1p2p3}, since
 $\left\|g\right\|^{2}_{\Hw} =\sum_{n}\left|b_n\right|^2 w^{-1}_{n}\asymp\sum_{n}\left|b_n\right|^2
\asymp  S<\infty$.

 By Lemma \ref{homogeneous in each Hp}, $g$ is in $\cap_{0<p<\infty}\Hl^{p}_{w}.$

It remains to prove that $T_g $ is unbounded on $\Hw$. We again choose as test functions (cf the proof of Proposition \ref{homo bd})
$$ f_x(s):= \prod_{\frac{x}{2}<p\leq x}\left(1+w^{1/2}_{2}p^{-s}\right)=\sum_{n\geq 1}a_n n^{-s}.$$
 $S_x$ is the set of square free integers generated by $\frac{x}{2}<p\leq x$. Set $V_x=\left\{n\in S_x,\ \omega(n)\geq\frac{N(x)}{2}\right\}$.

For $n\in V_x$, set
$$  A_n:=\sum_{p_1p_2p_3|n}b_{p_1p_2p_3}\left(\lo(p_1p_2p_3)\right)a_{\frac{n}{p_1p_2p_3}}$$
 The coefficients in $A_n$ satisfy
$$ b_{p_1p_2p_3} \left(\lo(p_1p_2p_3)\right)
\gtrsim \frac{\lo x}{x^{3/2}\left(\lo_2 x\right)^{\e+\d+1}}.$$
Since $\left\|f_x\right\|^{2}_{\Hw}\asymp\left|V_x\right|$, we see that 
\begin{align*}
  \left\|T_g f_x\right\|^{2}_{\Hw}&\geq\sum_{n\in V_x}w^{-1}_{n}\left(\lo n\right)^{-2}A^{2}_{n}\\
		&\gtrsim\sum_{n\in V_x}w^{-\omega(n)}_{2}\left(\omega(n)\lo x\right)^{-2}\left[ \frac{\lo x}{x^{3/2}\left(\lo_2 x\right)^{\e+\d+1}}\binom{\omega(n)}{3}\left(w^{1/2}_{2}\right)^{\omega(n)-3}\right]^2\\
		&\gtrsim \left\|f_x\right\|^{2}_{\Hw} \left(\frac{x}{\lo x}\right)^4 \frac{1}{x^3\left(\lo_2 x\right)^{2(\d+\e+1)}},
\end{align*}
and the proof is complete. 
 \end{proof}
\section{Comparison of $\Xw$ with other spaces of Dirichlet series}\label{compar}
The  previous results  enable us to derive some inclusions involving $\Xw$.

In the context of the unit disk, the space of symbols    $g$ for which the Volterra operator $J_g$ (\ref{volterra D})  is bounded on $A^2_{\alpha}(\D)$ is 
 $ \text{Bloch}(\D)$. It coincides with the space of holomorphic $g$ such that the Hankel form (\ref{Hf Dd}) is bounded, and with the dual space of $A^1_{\alpha}(\D)$.

  We shall study the counterparts of these facts for $\Xw$.

\subsection{Bounded Hankel forms}
The Hilbert space $\Hw$ is equipped with the inner product $\left\langle .,.\right\rangle_{\Hw}$. The Hankel form of symbol $g\in\Dl$ is defined on $\Hw$ by
\begin{equation}\label{Hf Hw}
	H_g(fh):=\left\langle fh,g\right\rangle_{\Hw}.
\end{equation}
We say that $H_g$ is bounded on $\Hw\times \Hw$ if there is a constant $C$ such that 
\begin{equation*}
	\left|H_g(fh)\right|\leq C\left\|f\right\|_{\Hw}\left\|h\right\|_{\Hw}\ \text{ for }f,h\in \Hw.
\end{equation*}
The weak product $\Hw\odot\Hw$ is the Banach space defined as the closure of all finite sums $F=\sum_{k}f_k h_k$, where $f_k, h_k\in \Hw$, under the norm 
\begin{equation*}
	\left\|F\right\|_{ \Hw\odot \Hw}:=\inf\sum_{k}\left\|f_k\right\|_{\Hw}\left\|h_k\right\|_{\Hw} .
\end{equation*}
Here the infimum is taken over all finite representations of $F$ as $F=\sum_{k}f_k h_k$.

Let $\Yl$ be a Banach space of Dirichlet series in which the space of Dirichlet polynomials $\Pl$ is dense. We say that a Dirichlet series $\phi$ is in the dual space $\Yl^{*}$ if the linear functional induced by $\phi$  via the $\Hw$-pairing is bounded. In other words, $\phi\in \Yl^{*}$ if and only if 
$$ v_{\phi}(f)=\left\langle f,\phi\right\rangle_{\Hw} ,\ f\in\Pl,$$
extends to a bounded functional on $\Yl$.

From 
its definition, $H_g$ (\ref{Hf Hw}) is bounded on $\Hw$ if and only if $g\in \left(\Hw\odot \Hw\right)^{*}$.

We aim to relate Hankel forms and Volterra operators. The primitive of $f\in \Dl$ with constant term $0$ is denoted by
\begin{equation*}
	\p ^{-1}f(s):=-\int^{+\infty}_{s}f(u)du,
\end{equation*}
 We observe that
	$$H_g(fh)=f\left(+\infty\right)h\left(+\infty\right)g\left(+\infty\right)+\left\langle \p ^{-1}(f'h),g\right\rangle_{\Hw}+\left\langle \p ^{-1}(fh'),g\right\rangle_{\Hw}.$$
The 
  Banach space $\p ^{-1}\left(\p\Hw\odot \Hw \right)$  is the completion of the space of Dirichlet series $F$ whose derivatives have a finite sum representation $F'=\sum_{k}f_k h'_k$, under the norm
$$ \left\|F\right\|_{\p ^{-1}\left(\p\Hw\odot  \Hw \right)}:=\left|F(+\infty)\right|+\sum_{k}\left\|f_k\right\|_{\Hw}\left\|h_k\right\|_{\Hw},  $$
where the infimum is taken over all finite representations. The product rule $(fg)'=f'g+fg'$ implies that 
$$	\Hw\odot  \Hw\subset \p ^{-1}\left(\p\Hw\odot  \Hw \right),$$
and then 
\begin{equation}\label{inclu dual  X otimes X}
\left(\p ^{-1}\left(\p\Hw\odot  \Hw \right)\right)^{*}	\subset \left(\Hw\odot \Hw\right)^{*}.    
\end{equation}
It has been shown in \cite{bre-perf} that, for the space $\Hl^2_{0}=\left\{f\in \Hl^2\ :\ f(+\infty)=0\right\}$, the inclusion  $\left(\p ^{-1}\left(\p\Hl^2_{0}\odot  \Hl^2_{0} \right)\right)^{*}	\subset \left(\Hl^2_{0}\odot \Hl^2_{0}\right)^{*}$ is strict. As for the space $\Hw$, the question whether the inclusion is strict remains open.

The membership of  $g$ in  $\left(\p ^{-1}\left(\p\Hw\odot  \Hw \right)\right)^{*}$ is equivalent to the boundedness of the so-called "half-Hankel  form"
\begin{equation}\label{half H}
 (f,h)\mapsto \left\langle \p ^{-1}(f'h),g\right\rangle_{\Hw}.
\end{equation}

As in the case of $\Hl^2$, the boundedness of $T_g$ implies  the boundedness of $H_g$.

\begin{thm}\label{Tg bd implies Hg bd}
If the Volterra operator $T_g$ is bounded on $\Hw$, then the Hankel form $H_g$ is bounded.
\end{thm}

\begin{proof}
We adapt the proof of Corollary 6.2 in \cite{bre-perf-seip} to the framework of the polydisk $\D^{\infty}$. Polarizing the Littlewood-Paley formula  (\ref{PLf}), we get
$$	\left\langle f,g\right\rangle_{\Hw}=f(+\infty) g(+\infty)+4\int_{\D^{\infty}}\int_{\R}\int^{+\infty}_{0}f'_{\chi}(\sigma+it)\overline{g'_{\chi}(\sigma+it)}\sigma d\sigma\frac{dt}{1+t^2}d\mu_w(\chi).$$
Then, we derive an expression  of the half-Hankel form
$$ \left\langle \p ^{-1}(f'h),g\right\rangle_{\Hw}=4 \int_{\D^{\infty}}\int_{\R}\int^{+\infty}_{0}f'_{\chi}(\sigma+it)h_{\chi}(\sigma+it)\overline{g'_{\chi}(\sigma+it)}\sigma d\sigma\frac{dt}{1+t^2}d\mu_w(\chi).$$
Since $T_g$ is bounded on $\Hw$, the Carleson measure characterization (\ref{CM Tg Y}) induces that the form (\ref{half H}) is also bounded. Then $H_g$ is bounded on $\Hw\odot \Hw$ by the inclusion (\ref{inclu dual  X otimes X}).
\end{proof}
The previous Theorem states that
 we have   $$\Xw\subset \left(\Hw\odot  \Hw\right)^{*}.$$
The rest of the section is devoted to study the reverse  inclusion.

Let $l^{2}_{w}$ denote the Hilbert space of complex sequences $a=(a_n)_n$ such that
$$\left\|a\right\|_{l^{2}_{w}} :=\left(\sum_{n\geq 1}\frac{\left|a_n\right|^2}{w_n}\right)^{1/2} <\infty.$$
A sequence $(\rho_n)_n$ generates the following multiplicative Hankel form 
\begin{equation}\label{form rho}
	\rho(a,b):=\sum^{+\infty}_{n=1}\sum^{+\infty}_{m=1}a_mb_n\frac{\rho_{mn}}{w_{mn}},\ a,b\in l^{2}_{w}.
\end{equation}

The symbol of the form is the Dirichlet series  $g(s)=\sum_{n\geq 1}\overline{\rho_n} n^{-s}$.
The form $ \rho$ is said to be bounded if there is a constant $C$ such that
$$ \left|\rho(a,b)\right|\leq C\left\|a\right\|_{l^{2}_{w}} \left\|b\right\|_{l^{2}_{w}}. $$

If $f$ and $h$ are  Dirichlet series with coefficients $a$ and $b$, respectively, we have 
$$  H_{g}(fh)=\left\langle fh,g\right\rangle_{\Hw}= \rho(a,b).$$
When the symbol $g$ has non negative coefficients, there is equivalence between the boundedness of $H_g$ and the half-Hankel form (\ref{half H}). In fact, the proof given for $\Hl^2$ in \cite{bre-perf} is valid for the spaces $\Hw$.

\begin{prop}\label{Hf positive coeff}
Let 
 $g(s)=\sum_{n\geq 1}\overline{\rho_n} n^{-s}$ be in $\Hw$. The linear functional defined on $\Hw$
$$v_g(f) :=\left\langle f,g\right\rangle_{\Hw} $$
is bounded on $\p ^{-1}\left(\p\Hw\odot  \Hw \right)$ if and only if the weighted form 
$$ J_g(a,b)=\sum^{+\infty}_{n=1}\sum^{+\infty}_{m=1}a_m b_n\frac{\log n}{\log m+ \log n}\frac{\rho_{mn}}{w_{mn}}, $$
(where it is understood that for $m=n=1$, the summand is $0$) is bounded on $l^{2}_{w}\odot  l^{2}_{w}$. 
The norms are equivalent, i.e.
$$ \left\|g\right\|_{\left(\p ^{-1}\left(\p\Hw\odot  \Hw \right)\right)^{*}}\asymp \left\|v_g\right\|\asymp \left|\rho_1\right|+\left\|J_g\right\|. $$
If $\rho_k\geq 0$ for all $k$, then $g\in\left(\p ^{-1}\left(\p\Hw\odot  \Hw \right)\right)^{*}$ if and only if $g\in\left(\Hw\odot  \Hw\right)^{*}$, with equivalent norms.
\end{prop}
Proposition \ref{Hf positive coeff} will enable us to provide examples of symbols $g$ for which the Hankel form $H_g$ and the half-Hankel form (\ref{half H}) are bounded, but the Volterra operator $T_g$ is unbounded (see the proof of Proposition \ref{inclu strict}). This differs from the case of  weighted Dirichlet spaces on the unit disk, for which the boundedness of $H_g$, the form (\ref{half H}) and $T_g$ are equivalent \cite{ale-per}.

For convergence reasons, we will consider Hankel forms defined on Dirichlet series without constant term. So we will work on the space
$$\Hl^2_{w,0}=\left\{f\in \Hl^2_w\ :\ f(+\infty)=0\right\}.$$

We have seen in Lemma \ref{embedding}  that the space $\Hw$ is embedded in a Bergman space of the form $A_{i,\delta}\left(\C_{1/2}\right)$. For $\delta>0$, it is thus natural to define the Hankel form 
\begin{equation}\label{H delta}
	H^{\left(\delta\right)}(fh):=\int^{+\infty}_{1/2}f(\sigma)h(\sigma)\left(\sigma-\frac{1}{2}\right)^{\delta}d\sigma,\ f,h\in \Hw_{,0}.
\end{equation}
Such multiplicative forms have been considered in the context of  $\Hl^2$ \cite{bre-perf-seip-sis}  and on $\Al^{2}_{1}$ \cite{bayart-haimi}.

Since $K^{\Hw}(s,u)-1=\sum_{n\geq 2}w_n n^{-\overline{u}}n^{-s}$ is the reproducing  kernel of $\Hw_{,0}$, we see that $H^{\left(\delta\right)}(fh)=\left\langle fh,\phi_{\delta}\right\rangle_{\Hw} $, where
$$	\phi_{\delta}(s)=\int^{+\infty}_{1/2}\left[K^{\Hw}(s, \sigma)-1\right]\left(\sigma-\frac{1}{2}\right)^{\delta}d\sigma=\sum^{+\infty}_{n=2}\frac{w_n}{\sqrt{n}\left(\log n\right)^{\delta+1}}n^{-s}.$$

\begin{prop}\label{H delta bd}
Let $\delta>0$ as in 
(\ref{def delta}).
 Then   $H^{\left(\delta\right)}$ defined in (\ref{H delta}) is a multiplicative Hankel form with symbol  $\phi_{\delta}$, which is bounded on $\Hw_{,0}\odot  \Hw_{,0}$.
\end{prop}

\begin{proof}
The proof is similar to that of Theorem 13 in \cite{bayart-haimi}. The Cauchy-Schwarz inequality ensures that 
$$ \left|H^{\left(\delta\right)}(fh)\right|\leq\left(\int^{+\infty}_{1/2}\left|f(\sigma)\right|^2\left(\sigma-\frac{1}{2}\right)^{\delta}d\sigma\right) ^{1/2}
\left(\int^{+\infty}_{1/2}\left|h(\sigma)\right|^2\left(\sigma-\frac{1}{2}\right)^{\delta}d\sigma\right) ^{1/2}.$$
If $f(s)=\sum^{+\infty}_{n=2}a_n n^{-s}$, notice the pointwise estimate
$$ \left|f(\sigma)\right|^2\leq\left\|f\right\|^2_{\Hw} \left(\sum^{+\infty}_{n=2}w_n n^{-2\s}\right)\lesssim \left\|f\right\|^2_{\Hw}4^{-\sigma},\ \text{ for }\sigma\geq 1.$$ 
Since the bounded measure $d\mu(\sigma+it)=\chi_{)1/2, 1]}(\sigma)\left(\sigma-\frac{1}{2}\right)^{\delta}d\sigma$, supported on the real line, is Carleson for $A_{i,\delta}\left(\C_{1/2}\right)$,  $\mu$ is Carleson for $\Hw $ by Lemma \ref{CM B alpha A beta}, and
\begin{align*}
\int^{+\infty}_{1/2}\left|f(\sigma)\right|^2\left(\sigma-\frac{1}{2}\right)^{\delta}d\sigma&=\left(\int^{1}_{1/2}+\int^{+\infty}_{1}\right)\left|f(\sigma)\right|^2\left(\sigma-\frac{1}{2}\right)^{\delta}d\sigma \lesssim\left\|f\right\|^2_{\Hw}.
\end{align*}
\end{proof}

We next
exhibit symbols   giving rise to  bounded Hankel forms and  bounded half-Hankel forms, though the associated Volterra operator is unbounded.

\begin{prop}\label{inclu strict}
  We have the strict inclusions
	\begin{align*}
	\Xl(\Hw_{,0})& \subset_{\neq}\left(\Hw_{,0}\odot   \Hw_{,0}\right)^{*};\\
	\Xw &\subset_{\neq}\left(\Hw\odot   \Hw\right)^{*}.
	\end{align*}

\end{prop}

\begin{proof}
It just remains to check the strictness  of the inclusions. For the exponent $\delta=\delta(w)$ and $\frac{1}{2}\leq a<1$, consider the symbol in  $\Hw_{,0}$
$$ g(s)=\sum^{+\infty}_{n=2}\frac{w_n}{{n}^a\left(\log n\right)^{\delta+1}}n^{-s}.$$
From Proposition \ref{H delta bd} and the fact that the coefficients are positive,   $g$ is in $\left(\Hw_{,0}\otimes  \Hw_{,0}\right)^{*}$ for any  $\frac{1}{2}\leq a<1$. In fact, the half Hankel form corresponding to $g$ is bounded. We have seen in Proposition \ref{primitive Zw} that $T_g$ is not bounded on $\Hw.$ Since $T_g 1=g$, $g$ does not belong to $\Xl(\Hw_{,0})$.

In order to prove that $g\in \left(\Hw\odot   \Hw\right)^{*}$, we consider the associated multiplicative form $\rho$ (\ref{form rho}). Let $f,h$ be Dirichlet series with coefficients $a,b$, belonging to $\Hw$. Since
\begin{align*}
\rho(a,b)&=\sum_{m,n\geq 2}a_m b_n\frac{\rho_{mn}}{w_{mn}}+a_1\sum^{+\infty}_{n=1}b_n\frac{\rho_{n}}{w_{n}}+b_1\sum^{+\infty}_{m=1}a_m\frac{\rho_{m}}{w_{m}}\\
&=H_g\left(\left(f-f\left(\infty\right)\right) \left(g-g\left(\infty\right)\right)\right)+f\left(\infty\right)\left\langle h,g\right\rangle_{\Hw}+g\left(\infty\right)\left\langle f,g\right\rangle_{\Hw},
\end{align*}
the first part of the proof entails that $H_g$ is bounded on $\Hw\odot   \Hw$.
\end{proof}

\subsection{$\Xw$ and the dual of $\Hl^{1}_{w}$}
Keeping in mind the results known for Bergman spaces of the unit disk, it is natural to compare $\Xw$ and  $\left(\Hl^{1}_{w}\right)^{*}$. 

In general, the dual of $\Hl^{1}_{w}$ is not known. However, it is shown in \cite{bayart-haimi} that 
\begin{equation}\label{mathcal K inclu}
	\Kl\subset \left(\Al^{1}_{1}\right)^{*},
\end{equation}
where $\Kl$ is the space of Dirichlet series $f(s)=\sum^{+\infty}_{n=1}a_n n^{-s}$ such that 
$$ \sum^{+\infty}_{n=1}\frac{d_4(n)}{\left[d(n)\right]^2}\left|a_n\right|^2<\infty. $$
The following consequence of this inclusion will  stress upon the difference between the finite and infinite dimensional setting.

\begin{prop}
$\left(\Al^{1}_{1}\right)^{*}$ is not contained in $\Xl\left(\Al^{2}_{1}\right)$.
\end{prop}

\begin{proof}
By Abel summation and the Chebyshev estimate,  the symbol 
$$  g(s)=\sum^{+\infty}_{n=2}\frac{d(n)}{n^{a}(\log n)^2} n^{-s},\text{ for }\frac{1}{2}< a <1,$$
is in $\Kl$, and thus in $\left(\Al^{1}_{1}\right)^{*}$. However,   $T_g$ is unbounded on $\Al^{2}_{1}$ (Proposition \ref{primitive Zw}).
\end{proof}

\subsection{$\Xw$ and the spaces $\Hl^{p}_{w}$}
It has been shown  in \cite{bre-perf-seip} that $BMOA(\C_0)\cap\Dl\subset_{\neq} \Xl(\Hl^2) \subset_{\neq}\cap_{0<p<\infty}\Hl^{p}.$ We have an analogue for Bergman spaces of Dirichlet series.
\begin{thm}\label{inclu X(Hw)}
 We have the strict inclusions
$$BMOA(\C_0)\cap\Dl\subset_{\neq} \Xw \subset_{\neq}\cap_{0<p<\infty}\Hl^{p}_{w}.$$
\end{thm}

\begin{proof}
The inclusions have been proved in Theorem \ref{boundedness hilbert} and Corollary \ref{Tg bd implies g Lp}. As observed in \cite{bre-perf-seip}, the symbols 
$g(s)=\sum^{+\infty}_{n=2}\frac{\psi(n)}{\log n}n^{-s}$, where $\psi$ is the completely multiplicative function defined on the primes by $\psi(p):=\lambda p^{-1}\log p$, $0<\lambda\leq 1$, are in $\Xl(\Hl^2)$, and satisfy
$$\sum^{+\infty} _{n=1}\psi(n)n^{-\sigma}\asymp \exp\left(\lambda \sum_p\frac{\log p}{p^{1+\sigma}}\right)\asymp \exp\left(\lambda \frac{1}{\sigma}\right), \ \sigma>0.$$
Hence, they are not in $BMOA(\C_0)$, though they belong to $\Xw$ (Lemma \ref{bd H2 implies bd Hw}).

The second inclusion is strict by Proposition \ref{ctex m=3}.
\end{proof}

With the  method of Proposition \ref{primitive Zw}, one can show that $g(s)=\sum_{n\geq 2}\frac{n^{-a}}{\log n}n^{-s}$, $1/2\leq a<1$, is not in $ \Xw$, though it belongs to $BMOA(\C_{1-a})$ \cite{bre-perf-seip}. Therefore, we have the strict inclusion
$$   \Xw\subset_{\neq} \text{Bloch}(\C_{1/2}).$$

\subsection{ $\Xw\cap\Dl_{d}$ and Bloch spaces }
\begin{thm}\label{inclu X(Hdw)}
Let $d$  be a positive integer.
 The following inclusions hold
$$\Dl_{d}\cap \text{Bloch}(\C_0)\subset \Dl_{d}\cap \Xw \subset_{\neq}\Bl^{-1}\text{Bloch}(\D^d).$$
\end{thm}

\begin{proof}
The first inclusion has been shown in Theorem \ref{boundedness hilbert} (a).

If $g$ is in $\Dl_{d}\cap \Xw $, Theorem \ref{Tg bd implies Hg bd}  implies that $H_g$ is bounded on $\Hw$. Therefore, the form $H_{\Bl g}$ (\ref{Hf D}) is bounded on the Bergman space $H^{2}_{w}(\D^d)$. From \cite{constantin50}, $\Bl g$ is in   $\text{Bloch}(\D^d).$

Here is  a function $g$ which is not in $\Xw$, such that $\Bl g$ is in  $\text{Bloch}(\D^2).$ Suppose that
$$  g'(s)= \frac{1}{1-2^{-s}}\log\left(\frac{1}{1-3^{-s}}\right),\ s\in\C_0.$$
Straightforward computations show that
$\Bl g\in \text{Bloch}(\D^2).$ The norms $\left\|.\right\|_{A^{2}_{\beta}(\D^2)}$ and $\left\|.\right\|_{B^{2}_{\beta}(\D^2)}$  being equivalent, our setting will be the space $A^{2}_{\beta}(\D^2)$. Now, for
$$  F(z)=\sum^{\infty}_{n=1}\frac{\left(n+1\right)^{\frac{\beta-1}{2}}}{\log(n+1)}z^n=\sum^{\infty}_{n=0}a_n z^n,\ z\in\D,$$
 define $f(s)=F(2^{-s}) F(3^{-s})$, for $s\in\C_0$. We have
$$ \left\|f\right\|^{2}_{\Hw}= \left\|F\right\|^{4}_{A^{2}_{\beta}(\D)}\asymp\left(\sum^{\infty}_{n=1}\frac{1}{\left(n+1\right)\left(\log(n+1)\right)^2}\right)^2<\infty.$$
Putting 
\begin{align*}
h_1(z_1)&=F(z_1)\frac{1}{1-z_1}=\sum^{\infty}_{m=0}A_m z^{m}_{1},\ z_1\in\D,\\
h_2(z_2)&=F(z_2)\log\left(\frac{1}{1-z_2}\right)=\sum^{\infty}_{n=0}B_n z^{n}_{2},\ z_2\in\D,
\end{align*}
we have $A_m \gtrsim  \frac{\left(m+1\right)^{\frac{\beta+1}{2}}}{\log(m+1)}$ and $B_n\gtrsim \left(n+1\right)^{\frac{\beta-1}{2}}$.
Therefore,
\begin{align*}
\left\|T_g f\right\|^{2}_{\Hw}&=\left\|R^{-1}\left(h_1 h_2\right)\right\|^{2}_{A^{2}_{\beta}(\D^2)}\asymp\sum_{m,n\geq 1}\frac{\left|A_m\right|^2 \left|B_n\right|^2}{(m+n+1)^2(m+1)^{\beta} (n+1)^{\beta}}\\
&\gtrsim \sum_{m\geq 1} \frac{m+1}{\left(\log(m+1)\right)^2}\frac{\log(m+1)}{(m+1)^2}=\sum_{m\geq 1}\frac{1}{(m+1)\log(m+1)}=+\infty,
\end{align*}
which proves the claim.
\end{proof}

A consequence of Theorems \ref{boundedness hilbert} and \ref{inclu X(Hw)} is that
$$ \text{Bloch}(\C_0)\cap \Dl_d\subset \cap_{0<p<\infty}\Hl^{p}_{d,w}.$$
This inclusion can be viewed as a counterpart of the situation of the disk, where $\text{Bloch}(\D)\subset \cap_{0<p<\infty}A^{p}_{\b}(\D).$

\end{document}